\documentclass[11pt]{amsart}

\usepackage{amssymb,verbatim}
\usepackage{amsmath,amsfonts}
\usepackage[mathscr]{euscript}
\usepackage{amsthm}
\usepackage[english]{babel}
\usepackage{latexsym}
\usepackage{amscd, epsfig}
\usepackage{graphicx}
\usepackage{tikz}
\newcommand{\tancirc}{\mathchoice
 {\raisebox{0.06em}{\tikz[scale=0.095]{\draw[line width=0.45pt] (0,0) circle (1); \draw[line width=0.45pt] (0.45,0) circle (0.55);}}}
 {\raisebox{0.06em}{\tikz[scale=0.095]{\draw[line width=0.45pt] (0,0) circle (1); \draw[line width=0.45pt] (0.45,0) circle (0.55);}}}
 {\raisebox{0.03em}{\tikz[scale=0.07]{\draw[line width=0.35pt] (0,0) circle (1); \draw[line width=0.35pt] (0.45,0) circle (0.55);}}}
 {\raisebox{0.02em}{\tikz[scale=0.056]{\draw[line width=0.3pt] (0,0) circle (1); \draw[line width=0.3pt] (0.45,0) circle (0.55);}}}}
\usetikzlibrary{arrows.meta,decorations.pathreplacing,calc}
\usepackage{ifthen}
\usepackage{url}
\usepackage{enumitem}
\usepackage{hyperref}
\hypersetup{colorlinks}
\usepackage{bbm}
\usepackage{bm}
\usepackage{mathrsfs}
\usepackage{cancel}

\usepackage{hyperref}
\hypersetup{colorlinks}
\RequirePackage{cleveref}
\usepackage{hypcap}

\definecolor{darkblue}{rgb}{0.0,0,0.7}
\newcommand{\darkblue}{\color{darkblue}}

\definecolor{darkred}{rgb}{0.68,0,0}

\definecolor{darkgreen}{rgb}{0,.38,0}

\definecolor{magenta}{rgb}{.51, 0, .51}

\hypersetup{colorlinks=true, citecolor=darkblue, linkcolor=darkblue}

\newcommand{\defn}[1]{\emph{\darkblue #1}}
\newcommand{\defna}[1]{\emph{\darkblue #1}}

\setlist[enumerate]{
	label=\textnormal{({\roman*})},
	ref={\roman*}}

\newcommand{\stepseqdiagram}[2]{%
\begin{tikzpicture}[x=1.12cm,y=0.42cm,>={Stealth[length=4.5pt]},baseline]
  \foreach \y in {2,4,6,8} {\draw[gray!14] (-0.15,\y) -- (8.15,\y);}
  \foreach \y in {1,...,8} {\node[gray!55,font=\tiny] at (-0.5,\y) {\y};}
  \foreach \x/\s/\e in {#1} {
     \ifnum\s=\e
        \draw[->,thick] (\x+0.05,\s) -- (\x+0.9,\e);
     \else
        \draw[->,thick] (\x+0.09,\s) -- (\x+0.9,\e);
     \fi
     \fill (\x+0.05,\s) circle (1.2pt);
     \node[font=\scriptsize\sffamily] at (\x+0.5,-1.5) {$a_{\s\e}$};
     \node[gray!65,font=\tiny] at (\x+0.5,-0.55) {\s};
  }
  #2
\end{tikzpicture}}

\newcommand{\stepseqswap}[4]{%
\begin{tikzpicture}[x=1.05cm,y=0.4cm,>={Stealth[length=4pt]},baseline]
  \foreach \y in {2,4,6,8} {\draw[gray!14] (-0.15,\y) -- (8.15,\y);}
  \foreach \y in {1,...,8} {\node[gray!55,font=\tiny] at (-0.5,\y) {\y};}
  \foreach \x/\s/\e in {#1} {
     \ifthenelse{\x=#3 \OR \x=#4}{\def\col{red!75!black}}{\def\col{black}}
     \ifnum\s=\e
        \draw[->,thick,\col] (\x+0.05,\s) -- (\x+0.9,\e);
     \else
        \draw[->,thick,\col] (\x+0.09,\s) -- (\x+0.9,\e);
     \fi
     \fill[\col] (\x+0.05,\s) circle (1.2pt);
     \node[font=\scriptsize\sffamily] at (\x+0.5,-1.5) {$a_{\s\e}$};
     \node[gray!65,font=\tiny] at (\x+0.5,-0.55) {\s};
  }
  #2
\end{tikzpicture}}

\newcommand{\stepseqswapC}[4]{%
\begin{tikzpicture}[x=0.62cm,y=0.30cm,>={Stealth[length=6.4pt]},baseline]
  \foreach \y in {1,...,8} {\draw[gray!12] (-0.1,\y) -- (8.1,\y);}
  \foreach \y in {1,...,8} {\node[gray!55,font=\tiny] at (-0.5,\y) {\y};}
  \foreach \x/\s/\e in {#1} {
     \ifthenelse{\x=#3 \OR \x=#4}{\def\col{red!75!black}}{\def\col{black}}
     \ifnum\s=\e
        \draw[->,semithick,\col] (\x+0.06,\s) -- (\x+0.9,\e);
     \else
        \draw[->,semithick,\col] (\x+0.1,\s) -- (\x+0.9,\e);
     \fi
     \fill[\col] (\x+0.06,\s) circle (0.9pt);
     \node[font=\footnotesize\sffamily] at (\x+0.5,-1.7) {$a_{\s\e}$};
  }
  #2
\end{tikzpicture}}

\newcommand{\stepseqdiagramC}[2]{%
\begin{tikzpicture}[x=0.62cm,y=0.30cm,>={Stealth[length=3.2pt]},baseline]
  \foreach \y in {1,...,8} {\draw[gray!12] (-0.1,\y) -- (8.1,\y);}
  \foreach \x/\s/\e in {#1} {
     \ifnum\s=\e
        \draw[->,semithick] (\x+0.06,\s) -- (\x+0.9,\e);
     \else
        \draw[->,semithick] (\x+0.1,\s) -- (\x+0.9,\e);
     \fi
     \fill (\x+0.06,\s) circle (0.9pt);
     \node[font=\tiny\sffamily] at (\x+0.5,-1.6) {$a_{\s\e}$};
  }
  #2
\end{tikzpicture}}

\newcommand{\stepseqN}[3]{%
\begin{tikzpicture}[x=0.40cm,y=0.26cm,>={Stealth[length=3.4pt]},baseline]
  \foreach \y in {1,...,8} {\draw[gray!12] (-0.1,\y) -- (8.1,\y);}
  \foreach \x/\s/\e in {#1} {
     \ifthenelse{\equal{#3}{\x}}{\def\col{red!75!black}}{\def\col{black}}
     \ifnum\s=\e
        \draw[->,semithick,\col] (\x+0.06,\s) -- (\x+0.9,\e);
     \else
        \draw[->,semithick,\col] (\x+0.1,\s) -- (\x+0.9,\e);
     \fi
     \fill[\col] (\x+0.06,\s) circle (0.7pt);
  }
  #2
\end{tikzpicture}}

\newcommand{\stepseqNstepbold}[2]{%
\begin{tikzpicture}[x=0.40cm,y=0.26cm,>={Stealth[length=7.6pt]},baseline]
  \foreach \y in {1,...,8} {\draw[gray!12] (-0.1,\y) -- (8.1,\y);}
  \foreach \y in {1,...,8} {\node[gray!55,font=\tiny] at (-0.55,\y) {\y};}
  \foreach \x/\s/\e in {#1} {
     \ifnum\s=\e
        \draw[->,line width=1.1pt] (\x+0.06,\s) -- (\x+0.9,\e);
     \else
        \draw[->,line width=1.1pt] (\x+0.1,\s) -- (\x+0.9,\e);
     \fi
     \fill (\x+0.06,\s) circle (0.7pt);
  }
  #2
\end{tikzpicture}}

\newcommand{\stepseqNbold}[3]{%
\begin{tikzpicture}[x=0.40cm,y=0.26cm,>={Stealth[length=6.8pt]},baseline]
  \foreach \y in {1,...,8} {\draw[gray!12] (-0.1,\y) -- (8.1,\y);}
  \foreach \y in {1,...,8} {\node[gray!55,font=\tiny] at (-0.55,\y) {\y};}
  \foreach \x/\s/\e in {#1} {
     \ifnum\s=\e
        \draw[->,semithick] (\x+0.06,\s) -- (\x+0.9,\e);
     \else
        \draw[->,semithick] (\x+0.1,\s) -- (\x+0.9,\e);
     \fi
     \ifnum\x<#3
        \fill (\x+0.06,\s) circle (0.7pt);
     \else
        \fill (\x+0.06,\s) circle (1.45pt);
     \fi
  }
  #2
\end{tikzpicture}}

\newcommand{\stepseqNboldhl}[4]{%
\begin{tikzpicture}[x=0.40cm,y=0.26cm,>={Stealth[length=6.8pt]},baseline]
  \foreach \y in {1,...,8} {\draw[gray!12] (-0.1,\y) -- (8.1,\y);}
  \foreach \y in {1,...,8} {\node[gray!55,font=\tiny] at (-0.55,\y) {\y};}
  \foreach \x/\s/\e in {#1} {
     \ifthenelse{\equal{#4}{\x}}{\def\col{red!75!black}}{\def\col{black}}
     \ifnum\s=\e \def\xa{0.06} \else \def\xa{0.1} \fi
     \ifnum\x<#3
        \draw[->,line width=1.1pt,\col] (\x+\xa,\s) -- (\x+0.9,\e);
        \fill[\col] (\x+0.06,\s) circle (0.7pt);
     \else
        \draw[->,semithick,\col] (\x+\xa,\s) -- (\x+0.9,\e);
        \fill[\col] (\x+0.06,\s) circle (1.45pt);
     \fi
  }
  #2
\end{tikzpicture}}

\numberwithin{equation}{section}
\numberwithin{figure}{section}

\DeclareMathOperator{\cf}{{cf}}

\newcommand{\la}{\lambda}

\newcommand{\ts}{{\hskip.04cm}}

\newcommand{\rqq}{\text{rq}}

\newcommand{\zz}{\mathbb{Z}}

\newcommand{\sm}{\setminus}

\newcommand{\Om}{\Omega}

\newcommand{\ga}{\gamma}
\newcommand{\si}{\sigma}

\newcommand{\Z}{ \mathbb Z}

\newcommand{\C}{ \mathbb C}

\newcommand{\s}[1]{\mathbf{#1}}
\newcommand{\set}[1]{\{#1\}}

\DeclareMathOperator{\Cdet}{Cdet}
\DeclareMathOperator{\Mdet}{Mdet}
\DeclareMathOperator{\Vdet}{Vdet}
\DeclareMathOperator{\Cper}{Cper}
\DeclareMathOperator{\Mper}{Mper}
\DeclareMathOperator{\Vper}{Vper}
\DeclareMathOperator{\sign}{sign}
\DeclareMathOperator{\inv}{inv}

\DeclareMathOperator{\weight}{{weight}}
\DeclareMathOperator{\pweight}{{pweight}}
\DeclareMathOperator{\head}{{head}}
\DeclareMathOperator{\length}{{length}}
\DeclareMathOperator{\INV}{{INV}}

\makeatletter
\def\th@plain{%
	\thm@notefont{}
	\itshape 
}
\def\th@definition{%
	\thm@notefont{}
	\normalfont 
}
\makeatother

\newtheorem{thm}{Theorem}[section]
\newtheorem{lemma}[thm]{Lemma}

\newtheorem*{claim*}{Claim}
\newtheorem{cor}[thm]{Corollary}

\newtheorem{conv}[thm]{Convention}

\theoremstyle{definition}
\newtheorem{ex}[thm]{Example}
\newtheorem{rem}[thm]{Remark}

\newtheorem{Def}[thm]{Definition}
\newtheorem{indasm}[thm]{Inductive Assumption}

\theoremstyle{plain}
\numberwithin{figure}{section}
\numberwithin{equation}{section}

\makeatletter
\newcommand{\duline}[1]{\mathpalette\du@line{#1}}
\newcommand{\du@line}[2]{%
  \ooalign{$\m@th#1#2$\cr
    \kern0pt
    \cleaders\hbox{\kern1pt\vrule height-1.0ex depth1.2ex width1.6pt\kern1pt}\hfil
    \kern0pt\cr}}
\makeatother

\def\zz{\mathbb Z}

\def\cc{\mathbb C}

\def\qqq{\mathbb Q}

\def\ov{\overline}

\def\sm{\smallsetminus}

\def\Om{\Omega}

\def\la{\lambda}
\def\ga{\gamma}
\def\si{\sigma}

\def\cC{\mathcal C}
\def\cB{{\mathcal{B}}}

\def\cA{\mathcal A}

\def\cF{\mathscr F}

\def\cO{\mathcal O}

\def\cS{\mathcal S}

\def\<{\langle}
\def\>{\rangle}

\def\y{ {\text {\rm y}  } }

\def\Z{\zz}

\def\SL{ {\text {\rm SL} } }
\def\GL{ {\text {\rm GL} } }

\def\0{{\mathbf 0}}

\def\.{\hskip.06cm}
\def\ts{\hskip.03cm}
\def\nts{{\hspace{-.05cm}}}

\def\nts{\hskip-.04cm}

\def\bC{\emph{\textbf{\textbf{C}\ts}}}
\def\bCr{{\textbf{\textbf{C}\ts}}}

\newcommand{\pf}{\mathrm{pf}}
\newcommand{\rad}{\mathrm{rad}}

\def\.{\hskip.06cm}
\def\ts{\hskip.03cm}

\def\nin{\noindent}

\newcommand{\textsu}[1]{\textup{\textsf{#1}}}

\newcommand{\ComCla}[1]{\textup{\textsu{#1}}}
\newcommand{\sharpP}{\ComCla{\#P}}
\newcommand{\SP}{\ComCla{\#P}}

\newcommand{\GapL}{\ComCla{GapL}}

\newcommand{\NP}{\ComCla{NP}}

\renewcommand{\P}{\ComCla{P}}

\newcommand{\FP}{\ComCla{FP}}

\newcommand{\VP}{\ComCla{VP}}
\newcommand{\VBP}{\ComCla{VBP}}

\def\SP{\sharpP}
\def\poly{{\P}}

\def\bq{\textbf{{\textit{q}}}}

\newcommand{\genstirlingI}[3]{%
  \genfrac{[}{]}{0pt}{#1}{#2}{#3}%
}

\newcommand{\stirlingI}[2]{\genstirlingI{}{#1}{#2}}

\DeclareMathOperator{\Imm}{\textnormal{Imm}}

\title[Quantum determinants in polynomial time]
{Quantum determinants in polynomial time}

\author[\ts Igor Pak]{Igor Pak}
\address[Igor Pak]{Department of Mathematics, UCLA,  Los Angeles, CA 90095.}
\email{\texttt{pak@math.ucla.edu}}

\author[\ts Daniel Soskin]{Daniel Soskin}
\address[Daniel Soskin]{Department of Mathematics, UCLA,  Los Angeles, CA 90095.}
\email{\texttt{dsoskin@math.ucla.edu}}

\begin{document} \hskip.1cm

\begin{abstract}
We give an algebraic branching program of polynomial size
which computes Cayley determinant of right quantum matrices.
This is a rare example of an efficient computation of a
noncommutative determinant, and the first such example for
quantum groups.
We extend the results to the $q$-Cayley determinant
of $q$-right quantum matrices, as well as to their
multiparameter generalization.

The proofs are entirely combinatorial, as we relate
Cayley, Moore and Valiant determinants using bijections/involutions
on words.  We then employ the celebrated determinant construction
of Mahajan and Vinay (SODA'97), to obtain the results.
\end{abstract}

\maketitle


\section{Introduction} \label{s:intro}

\subsection{Foreword} \label{ss:fore}
The \emph{determinant} \ts is one of the most classical and well studied polynomials,
ever since it was introduced by Leibnitz in 1693, see e.g.\ \cite{Muir60}.
The \emph{noncommutative determinant}, first studied by Cayley
\cite{Cay45}, plays a similar role (see~$\S$\ref{ss:back-det}).
While computing commutative determinants is extremely well understood
(see~$\S$\ref{ss:back-als}),
%
computing noncommutative determinants is quite difficult, and remains
a fundamental problem in its own right.  In the complexity theoretic setting,
this problem was first considered 
by Hyafil \cite{Hya77}, building on the work of Winograd \cite{Win70}
and Strassen \cite{Str73}, and has attracted considerable attention
over the past few decades.

In a landmark paper \cite{Nis91}, Nisan established an \ts $\Om(2^n)$ \ts
lower bound for the size of the \emph{algebraic branching program} (ABP)
for computing the noncommutative $n\times n$ determinant over free algebra
\ts $\cc\<x_{11},\ldots,x_{nn}\>$.  This makes the noncommutative determinant
similar to the noncommutative permanent, for which Nisan proved the same
lower bound, and markedly different from the determinant over commutative
rings, where there is a polynomial size ABP, see e.g.\ \cite{Bur26}.

Motivated in part by the \emph{permanent vs.\ determinant paradigm} \ts introduced
by Valiant \cite{Val79}, much effort has been made to build on Nisan's paper.
Notably, Chien and Sinclair \cite{CS07} showed that the exponential lower bound
holds for matrices over the quaternion algebra, or over the algebra of \ts
$2\times 2$ matrices with entries in a field of characteristic zero
(among other examples).  They asked for which algebras does the determinant
requires large ABP size.

Since Nisan's ABP model is somewhat constrained, one can ask if there is
a polynomial size \emph{arithmetic circuit} \ts for the noncommutative determinant.
In a surprising discovery, Arvind and Srinivasan \cite{AS18} showed that
computing the noncommutative determinant is at least as hard as the permanent
(up to polynomial factors), and thus at least as hard as the commutative
permanent.
Soon after, Chien et al.~\cite{CHSS11} and Bl\"aser
\cite{Bla15} emulated Valiant's proof that $0/1$ permanent is $\SP$-complete
\cite{Val-perm}, to resolve the problem over which associative algebras \ts $\cA$ \ts
the determinant is hard. They showed that if \ts $\cA/\rad \ts \cA$ \ts is noncommutative,
then computing the determinant is \ts $\SP$-hard; it is in $\FP$
otherwise.\footnote{Here we are assuming that the algebra $\cA$ is defined over~$\qqq$.}
See also Gentry \cite{Gen14}, for a simpler approach via Barrington's theorem \cite{Bar86}.

Another motivation for studying noncommutative determinants lies in
their ability to approximate the $0/1$ permanent.  In a commutative setting,
this approach was introduced by Godsil and Gutman \cite{GG81} over the reals,
and further developed by Karmarkar et al.\ \cite{K+93} and by
Barvinok \cite{Bar99}, over the complex numbers and quaternions,
respectively.  The noncommutative version was further studied
in \cite{Bar00,CRS03,MR12}.

In a positive direction, computing noncommutative determinants clearly hit
a wall, with the progress stifled by the powerful negative results described above.
This is somewhat surprising, since there are extremely well-known noncommutative
associative algebras\footnote{All algebras in this paper are associative, so we will not
specify this in the future.}
used across representation theory and algebraic combinatorics, which are not
of the matrices-over-algebras type.

\smallskip

\subsection{Our results}\label{ss:intro-results}
In this paper, we prove that (noncommutative) determinants of quantum matrices
have a polynomial size ABP.\footnote{Over~$\qqq$, this implies that the determinants
can be computed in polynomial time.}
By a determinant we mean (the usual)
\defn{Cayley determinant} \ts of a matrix \ts $A=(a_{ij})$ \ts with
noncommutative entries:
\begin{equation}\label{eq:Cdet}
\Cdet(A)
\ := \ \sum_{\sigma \in S_n} \, (-1)^{\inv(\sigma)} \,
 a_{1 \sigma_1} \. \cdots \. a_{n \sigma_n}\..
\end{equation}
We define \emph{quantum matrices} \ts step-by-step, as they comes in flavors
of varying difficulty, and the most general results combine the flavors.

First, we consider \defn{Cartier--Foata {\rm (CF)} matrices} \cite{CF69,Foa65},
defined to have commuting elements in different rows (see a discussion
in~$\S$\ref{ss:def-basic} for our matrix notation):
\begin{equation}\label{eq:CF}
a_{\ell j} \. a_{ki} \ \,  = \ \, a_{ki} \. a_{\ell j}  \ \quad
 \text{for}  \quad i \ne j\ts.
\end{equation}
Even though the products in Cayley determinants of CF matrices have commuting
entries, the usual linear algebra technology no longer applies.

The celebrated approach by Mahajan and Vinay \cite{MV97,MV99},
gives an ABP to compute the determinant as a
commutative polynomial in matrix entries.  With some modifications
(see Remark~\ref{rem:MV-diff}), we show that this approach can be extended
to compute the Cayley determinant of CF matrices.  This is the simplest
of our results, see below.


Second, for a complex parameter \ts $q\ne 0$,
we consider \defn{$q$-CF matrices} \ts :
\begin{equation}\label{eq:qCF}
\left\{\aligned
 a_{\ell j} \. a_{ki} \ \  & = && \, a_{ki}\. a_{\ell j} \ \
 && \text{for}  \ \ i < j, \, k < \ell\ts, \\
 a_{\ell j} \. a_{ki} \ \ & = && \, q^2 \. a_{ki} \. a_{\ell j}
 \ \ && \text{for}  \ \ i < j, \, k > \ell\ts, \\
 a_{kj} \. a_{ki} \ \ & = && \, q \. a_{ki}\. a_{kj}
 \ \ && \text{for}  \ \ i < j\ts.
\endaligned\right.
\end{equation}
When \ts $q=1$ \ts we obtain the usual CF matrices.  
In this case, we compute the \ts \defn{$q$-Cayley determinant},
also called the \emph{quantum determinant}:
\begin{equation}\label{eq:qCdet}
\Cdet_q(A)
\, := \, \sum_{\sigma \in S_n} \, (-q)^{-\inv(\sigma)} \,
 a_{1 \sigma_1} \cdots a_{n \sigma_n}
\end{equation}
When \ts $q=-1$ \ts this is \defn{Cayley's permanent}.   See \cite{HK23} for the
hardness of computing \ts $\Cdet_q$ \ts at roots of unity in the commutative case.

Third, for an array \ts $\bq=(q_{ij})_{1 \le i < j \le n}$ \ts
of nonzero complex parameters, we define the  \defn{$\bq$-CF matrices} \ts
and the \defn{$\bq$-Cayley determinant}, giving a multiparameter quantum
deformations  of \eqref{eq:qCF} and \eqref{eq:qCdet}, respectively.  
See~$\S$\ref{ss:def-matrices},~$\S$\ref{ss:def-inv} for the definitions.  
%
%

Fourth and most important, we consider \defn{right-quantum {\rm (RQ)} matrices}~:
\begin{equation} \label{eq:rq}
\left\{\,\aligned
 a_{kj}\. a_{ki} \ \  & = && \, a_{ki}\. a_{kj} \ \ && \text{for}  \ \ i < j\ts, \\
 a_{ki}\. a_{\ell j} \, - \, a_{kj}\. a_{\ell i} \ \  &= && \, a_{\ell j}\. a_{ki} \, - \, a_{\ell i}\. a_{kj}
 \ \ &&\text{for}  \ \ i<j, \, k<\ell\ts.
\endaligned\right.
\end{equation}
Note that the RQ matrices generalize the CF matrices by introducing
new noncommutative relations on $2\times 2$ minors.  In this case,
we give an ABP construction to compute the Cayley determinant.

Next, we consider \defn{$q$-RQ matrices} \ts defined as a $q$-deformation
of relations \eqref{eq:rq}, and that also generalize the $q$-CF matrices:
\begin{equation} \label{eq:qRQ}
\left\{\,\aligned
 a_{kj}\. a_{ki} \ \  &=&& \, q\. a_{ki}\. a_{kj}\,, \ \ &&\text{for}  \ \quad i < j\ts,\\
 a_{kj}\. a_{\ell i} \, - \, q \. a_{ki}\. a_{\ell j} \ \  &=&& \, q^2 \ts a_{\ell i}\. a_{kj}
 \, - \,  q \. a_{\ell j}\. a_{ki}
 \ \ &&\text{for}  \ \quad i<j, \, k<\ell\ts.
 \endaligned\right.
\end{equation}
This algebra was introduced by Faddeev, Reshetikhin and Takhtajan \cite{FRT88} in
connection with quantum groups, and motivated by representation theory
and mathematical physics (see below).  When \ts $q=-1$, we obtain
\defn{antisymmetric right-quantum {\rm (ARQ)} matrices}.
Our general algorithm in this case gives unusual examples,
where the noncommutative permanent (but not necessarily the determinant!)
can be computed efficiently.

Finally, we consider the \defn{$\bq$-RQ matrices},
which give a multiparameter generalization of relations \eqref{eq:qRQ},
and which generalize the $\bq$-CF matrices.  See \eqref{eq:bq-RQ} for
the definition.  These matrices were introduced by Manin \cite{Man88,Man89},
motivated additionally by the ideas from noncommutative geometry,
in an effort to construct a noncommutative linear algebra.

\smallskip

\begin{thm}[{\rm Main theorem}{}] \label{t:main}
For all nonzero \ts $\bq=(q_{ij})$ \ts as above, there is an ABP of polynomial
size, which computes the $\bq$-Cayley determinant of \ts $\bq$-RQ matrices.
In particular, it computes the $($usual$)$ Cayley determinant of RQ matrices
and the Cayley permanent of ARQ matrices.
\end{thm}

\smallskip

This is a rare positive result for noncommutative determinants,
the first complexity result on quantum matrices in either direction,
and an unusual polynomial time result for the exact computation of the
permanent. An extensive background overview is given in the next section.

In the RQ case, the proof is based on two arguments in bijective combinatorics,
see Theorem~\ref{t:det-eq} for the main nonalgorithmic result of the paper.
First, we modify the Konvalinka--Pak bijection in \cite{KP07},
to show that the Cayley determinant is equal to the Moore determinant,
where the latter corresponds to factorizations of permutations into cycles.
Second, we build on the Mahajan--Vinay involution to show that Moore's
determinant is equal to Valiant's determinant, which is defined in
terms of \emph{clows} (closed walks), combined with algebraic properties
of RQ matrices.  In the $\bq$-RQ case, a great
deal of careful bookkeeping is also involved.

{\small
\begin{rem}\label{rem:matrices}
Some special cases of the theorem are easy to state and give striking
conclusions.  For example, one can take certain explicit $2\times 2$
matrices $a_{ij}$ which do not commute, but satisfy relations \eqref{eq:rq}.
See e.g.\ \cite{HL07-Jones}, where these matrices are given in connection
with the Jones polynomial in knot theory.  Our Main Theorem~\ref{t:main} in this
case gives a polynomial size ABP to compute the Cayley determinant of
such matrices.  This is in sharp contrast with the above mentioned result
by Chien and Sinclair \cite{CS07}, that to compute the Cayley determinant
of \emph{all} \ts such matrices one needs an exponential size ABP.
\end{rem}}

\medskip

\section{Background} \label{s:back}

\subsection{Noncommutative determinants}\label{ss:back-det}
In contrast with the commutative case, there are multiple notions of
noncommutative determinants of varying degree of generality and applicability.
The earliest definition \eqref{eq:Cdet} indeed goes back to Cayley \cite{Cay45},
and was motivated by Hamilton's work on quaternions.  To put it in historical
perspective, the \emph{Cayley--Hamilton theorem} was obtained only
in the 1850s.  We refer to \cite{Abe11} for the historical introduction.

Next, we have \defn{Moore's determinant} \ts defined by E.~H.~Moore \cite{Moo22},
in the context of Hermitian matrices over quaternions.  It is a summation over
all permutations given as an ordered product of cycles.  Moore's determinant will
play a crucial role throughout the paper, see $\S$\ref{ss:def-Moore} for the definition.
An extensive discussion of Moore's determinant in a historical context
is given in~\cite{Asl96}.

Other notable examples include the \emph{Dieudonn\'e determinant} \ts \cite{Die43}  for
matrices over skew fields, see e.g.\ \cite[$\S$20]{Dra83}, and \emph{Berezinian} \ts
for superalgebras\footnote{These are $\zz_2$-graded algebras describing supersymmetries,
originally studied in theoretical physics.} introduced by Berezin in~\cite{Ber87}.
More recently, Barvinok \cite{Bar00} introduced the \emph{symmetrized determinant} \ts 
with the goal of giving an approximation of the permanent, see also \cite{ABG26} for 
complexity aspects. 

The richest and most insightful are \emph{quasideterminants} \ts by
Gelfand and Retakh \cite{GR92,GGRW05}.  This is not a single object but
a family of rational functions in the division ring of free variables,
evaluated at noncommutative matrix entries, and related to each other by
certain ``homological relations''.  Although this definition unites many other
and underlies the philosophy of this paper, we will not need them
for our results, see however~$\S$\ref{ss:finrem-CH}

\smallskip

\subsection{Quantum matrices}\label{ss:back-quantum}
The CF matrices are sometimes called
\emph{row-commutative} \ts and \emph{partially commutative}.
They were introduced by Cartier and Foata \cite{Foa65,CF69}
in an attempt to understand, generalize and give a combinatorial
proof of the \emph{MacMahon's Master Theorem} (MMT), a classical result
in enumerative combinatorics, see e.g.\ \cite[$\S$1.2.11]{GJ83}.

Similarly, the $q$- and $\bq$-CF matrices were introduced by
Foata--Han \cite{FH07b} and by Konvalinka--Pak \cite{KP07},
respectively, as special cases of the $q$-RQ and $\bq$-RQ matrices.
The motivation was to give a combinatorial proof of the $q$-MMT
obtained earlier by Garoufalidis, L\^e and Zeilberger~\cite{GLZ06}.
Consequently, the $\bq$-MMT generalization and a bijective proof was
given in \cite{KP07}.  Another proof based on Koszul duality was given
in \cite{HL07}.  See also \cite{FH07a} for the first combinatorial
proof of the $q$-MMT.

The right-quantum (RQ) matrices form a Hopf algebra \ts $\cO_q(G)$ \ts
isomorphic to the coordinate ring of the algebraic group \ts $G=\GL(n,\cc)$.
As we mentioned above, $q$-RQ matrices were introduced in \cite{FRT88},
as a way to deform this Hopf algebra.  It is Hopf dual to the
quantum group construction by Drinfeld--Jimbo, which is a $q$-deformation
of the universal enveloping algebra \ts $U_q({{\mathfrak{gl}}})$.
This approach to define quantum matrices starting with the $R$-matrix
from the \emph{Yang--Baxter equation}, is now known as the
\emph{FRT construction}.

The quantum determinant \eqref{eq:qCdet} spans the center of the
\ts $\cO_q(G)$, so \ts $\cO_q(G)/(\Cdet_q=1)$ \ts gives a
$q$-deformation of the coordinate ring of \ts $\SL(n,\cc)$, ibid.
We refer to \cite{BG02,Kas95} for standard introductions to
the area, including the FRT construction. See also \cite{DF93} 
for a \emph{Gauss decomposition} \ts in this setting.  

Let us mention that both $q$-RQ matrices and the $q$-Cayley determinant
also appear in connection to the \emph{Geometric Complexity Theory} (GCT),
as shown by Blasiak, Mulmuley and Sohoni \cite{BMS15}.  They employed these tools
in an effort to advance the \emph{Kronecker coefficients problem} \ts that
is central in GCT, see e.g.\ an extensive discussion in \cite[$\S$6.6]{Aar16}.

The first connection between quantum determinants and
quasideterminants was given by Krob and Leclerc \cite{KL95}.
We refer to \cite{GR91,ER99} and \cite{KP07},
for explicit formulas relating the $q$- and $\bq$-Cayley determinants
with quasideterminants of $q$- and $\bq$-RQ matrices, respectively.
See also \cite[Rem.~1.3]{ESS00} for an advanced generalization. 

Computing $q$-Cayley determinants of $q$-RQ matrices has also important
applications in knot theory, as they determine the \emph{colored Jones
polynomial} \ts originally introduced in \cite{Jon87}.
Motivated by the $q$-MMT, it was proved by Huynh and L\^e \cite{HL07-Jones},
that this polynomial is the inverse of the $q$-Cayley determinant
of an almost $q$-RQ matrix with entries given by certain
$2\times 2$ matrices over $q$-differential operators
corresponding to crossing of the knot.  Another substitution
gives the Alexander polynomial (ibid.)  See also \cite{HL18} for
an implementation of this formula for knots with few crossings.

As we mentioned in the introduction, the $\bq$-RQ matrices
were defined by Manin in \cite{Man88,Man89}.
They are also called \emph{Manin} \ts and \emph{row-pseudo-commutative matrices},
and are extremely well-studied.
The \defn{$\bq$-Cayley determinant} \ts was defined in \cite[$\S$1.9]{Man89}
in terms of exterior powers, and expanded in terms of inversions in
\cite[$\S$3.5]{Man88}, see \eqref{eq:bq-Cdet}.
Manin also showed that it spans the center of the algebra of
$\bq$-RQ matrices. We refer  to \cite{CSS09,CFR09,CFRS14,JLZ24,Sil21}
for numerous algebraic and combinatorial identities for these matrices,
generalizing classical results in linear algebra.

\smallskip

\subsection{Algorithmic aspects}\label{ss:back-als}
Edmonds \cite{Edm67} was the first to rigorously analyze a version of the
Gaussian elimination algorithm over $\qqq$, and show that it runs in
polynomial time.  Kalorkoti \cite{Kal85} proved the lower bound $\Om(n^3)$
on the size of the arithmetic formula for computing the $n\times n$ determinant.

In \cite{Str73}, Strassen showed that how any such algorithm
can be made division-free, giving an algorithm for computing the
determinant in $O(n^5)$ steps,
see a discussion in \cite{BGH82}.\footnote{After several modern improvements,
the complexity of this algorithm drops to \ts $O(n^{3.37})$, see e.g.\ \cite{Urb10}.}
Berkowitz \cite{Ber84} gave the first explicit division-free algorithm to compute
the determinant.  Building on the work of Chistov and Valiant \cite{Chi85,Val92},
Mahajan and Vinay \cite{MV97} described a combinatorial algorithm compute the
determinant in time \ts $O(n^4)$.   We refer to \cite{MV99,Rote01} for historical
remarks and further references.

In \cite{Urb10}, Urba\'nska showed how the Mahajan--Vinay algorithm can
be sped up to \ts $O(n^{3.03})$, via a reduction to the matrix multiplication.
An especially simple division-free algorithm also based on matrix multiplication,
was given by Bird~\cite{Bird11}.  Most recently, starting from purely linear
algebraic considerations, Ikenmeyer \cite{Ike25} constructed a new ABP for
the determinant, and for the coefficients of the characteristic polynomial.

\smallskip

\subsection{Complexity aspects}\label{ss:back-cs}
Let \ts $\Bbbk$ \ts be a field, let
\ts $X=\{x_1,\ldots,x_k\}$ \ts be free variables, and let \ts $f\in \Bbbk\<X\>$ \ts
be a fixed polynomial of degree~$n$.
An \defn{ABP} is a directed acyclic graph \ts $G=(V,E)$ \ts with one \emph{source}~$S$ and
one \emph{sink}~$T$.  The vertices $v\in V$ are partitioned into levels $0,1,\ldots,n$,
with source a unique vertex at level $0$, and sink a unique vertex at level~$n$.
Let edges $v\to v'$ be labelled with a homogeneous linear form in~$X$.
For a path $\ga:S\to T$ in~$G$, denote by \ts $p_\ga$ the product of all
linear forms over edges in~$\ga$.  The \defn{size} of the ABP is the number
of vertices~$|V|$.

We say that $f$ is computed by the ABP, if \ts
$f=\sum_\ga p_\ga\ts$, where the summation is over all paths \ts $\ga:S\to T$.
More generally, for an algebra $\cA$ over~$\Bbbk$ generated by \ts $A=\{a_1,\ldots,a_k\}$,
an element \ts $D\in \cA$ \ts is computed by the ABP as above, if \ts $f(a_1,\ldots,a_k)=D$.
See e.g.\ \cite{Bur26} for an accessible introduction, and for the role of ABPs in
computing determinants.

The (noncommutative) arithmetic circuit is defined similarly, with operations
$(+)$ and $(\times)$ at each vertex, all of in-degree at most two, and without the
level restriction.  Obtaining (unconditional) superpolynomial circuit lower bounds
for a Cayley determinant remains a major open problem, see e.g.\ \cite[$\S$12.5]{Wig19}.
We refer to Hrube\v{s}, Wigderson and Yehudayoff
\cite{HWY11}, for a remarkable connection of this problem to a (commutative)
complexity of computing the sum-of-squares.

In \cite{HW15}, Hrube\v{s} and Wigderson extend noncommutative circuits for computation
of rational functions, including noncommutative matrix inversion closely related to
quasideterminants.  See also Garg et al.\ \cite{GGOW20}, for a
\emph{deterministic} \ts polynomial time algorithm deciding whether a symbolic matrix
in noncommutative variables is invertible over~$\qqq$.  Note that this is not the
same as deciding whether the Cayley determinant is nonzero, 
cf.~$\S$\ref{ss:finrem-R}.\footnote{Over division rings, being invertible is 
equivalent to having nonzero Dieudonn\'e determinant, see \cite{Die43,Dra83}.}

The algorithm of Mahajan and Vinay \cite{MV97} discussed above, can be viewed as an
ABP of size $O(n^3)$.  From the Boolean complexity point of view, it shows that
computing the determinant over~$\zz$ is in $\ts \GapL={\sf \#L}-{\sf \#L}$, and thus $\GapL$-complete.
Famously, Valiant showed that the determinant is complete in $\VP$ up to
quasipolynomial reductions \cite{Val79}.  Slightly relaxing the syntaxis from
formulas to \emph{weakly-skew arithmetic circuits}, gives an algebraic complexity
class $\ts\VBP$, where the determinant is complete up to polynomial reductions,
see \cite[$\S$2.5]{Bur26}.

Finally, we note that the \emph{permanent approximation problem} \ts is one
of the benchmarks in the area.  After much effort, a {\sf FPRAS} for the
problem was given by Jerrum, Sinclair and Vigoda in \cite{JSV04},
via the MCMC methods.  The simplicity of taking expectations of the
norms of random noncommutative determinants gives an alternative approach
that is yet to be proved effective.

\medskip


\section{Definitions and the setup} \label{s:def}

\subsection{Basic notation}\label{ss:def-basic}
We use \ts $[n]:=\set{1,\ldots,n}$, and let \ts $[n]^\ast$ \ts be the set
of finite sequences (words) in~$[n]$.  Similarly, we use \.
$\stirlingI{n}{2}=\big\{(i,j) \. : \. 1\le i < j\le n\big\}$.
We write all permutations $\sigma$ in the symmetric group~$S_n$ using
\emph{one-line notation}, as the word \ts $(\sigma_1,\ldots,\sigma_n)$.

In an important departure from standard conventions, we write matrix elements
in a transposed form, for example:
$$\begin{pmatrix}
    a_{11} & a_{21} & a_{31} \\
    a_{12} & a_{22} & a_{32} \\
    a_{13} & a_{23} & a_{33}
  \end{pmatrix}
$$

\medskip

\nin
The reason for this unfortunate choice is technical: it is to ensure that
the correspondence between words and paths is natural. For example,
we have the path \ts $1\to 3 \to 2\to 4$ \ts corresponds to the word
\ts $a_{13} \ts a_{32} \ts a_{24}$ \ts instead of the word \ts
$a_{31} \ts a_{23} \ts a_{42}\ts.$  Helpfully, almost all of the paper
can be read ignoring this change.\footnote{The reader who cannot stomach
this notation can pretend the matrix is given as usual, and switch the
meaning of rows/columns, with \emph{right-quantum} replaced by
\emph{left-quantum}.  In fact, Cayley in \cite{Cay45} did just that in
a similar context, leading to all sorts of confusion \cite{Abe11}.
We decided to avoid this approach, so to be able to quote unchanged 
the results from earlier papers.}

\smallskip

\subsection{Quantum matrices}\label{ss:def-matrices}
Denote by \. $\cF_n = \C\langle a_{ij} : 1 \le i, j \le n\rangle$ \. the
\emph{free associative algebra}, whose basis is the set of words in
the alphabet \ts $\{a_{ij}\}$, with concatenation as product.

Let \ts $I_{\cf}^q$ \ts be the two-sided ideals
in \ts $\cF_n$ \ts generated by the relations \eqref{eq:qCF}.
The quotient \ts $\cA_n^q=\cF_n/I^q_{\cf}$ \ts is called the
\defn{$q$-CF algebra}.  We write that two words \ts $u,v\in \cF_n$ \ts
are \emph{equal modulo} \eqref{eq:qCF} if they are equal in $\cA^q_n\ts$.

Similarly, let \ts $I_{\rqq}^q$ \ts be the two-sided ideals
in \ts $\cF$ \ts generated by the relations \eqref{eq:qRQ}.
The quotient \ts $\cB_n^q=\cF_n/I^q_{\rqq}$ \ts is called the
\defn{$q$-RQ algebra}.  We write that two words \ts $u,v\in \cF_n$ \ts
are \emph{equal modulo} \eqref{eq:qRQ} if they are equal in $\cB^q_n\ts$.


More generally, fix complex numbers \ts $q_{ij} \in \cc$, such that \ts $q_{ij} \ne 0$ \ts
for all $1 \le i < j \le n$.  Consider \defn{$\bq$-CF  matrices} \ts defined as
\begin{equation}\label{eq:bq-CF}
\smallskip
\left\{\, \aligned
 q_{k\ell} \, a_{\ell j}\. a_{ki} \ \  & = && \, q_{ij} \, a_{ki}\. a_{\ell j}\,
 \ \ &&\text{for} \ \ i < j, \ k < \ell\ts,\\
 a_{\ell j}\. a_{ki} \ \  & = && \, q_{ij} \. q_{\ell k} \, a_{ki}\. a_{\ell j}\,
 \ \ &&\text{for}
 \ \ i < j, \ k > \ell \ts,\\
 a_{kj}\. a_{ki} \ \ & = && \,q_{ij} \, a_{ki}\. a_{kj}\,,
 \ \ &&\text{for}  \ \ i < j\ts.
\endaligned\right.
\smallskip
\end{equation}

Similarly, consider \defn{$\bq$-RQ matrices} \ts defined as
\begin{equation} \label{eq:bq-RQ}
\smallskip
 \left\{\,\aligned
 a_{kj}\. a_{ki} \ \  & = && \, q_{ij}\, a_{ki}\. a_{kj}\,, \ \ && \text{for}  \ \ i < j\ts, \\
 a_{kj}\. a_{\ell i} \, - \, q_{ij}\, a_{ki}\. a_{\ell j} \,  & = &&
  \, q_{k\ell}\. q_{ij}\, a_{\ell i}\. a_{kj} \, - \,
  q_{k\ell}\, a_{\ell j}\. a_{ki}\,,
 \ \ && \text{for}  \ \ i<j, \ k<\ell\ts.
 \endaligned \right.
 \smallskip
\end{equation}

\nin
When all \ts $q_{ij}=q$ \ts we obtain $q$-CF and $q$-RQ matrices, respectively.
By analogy with the above, let \ts $I^\bq_{\cf}$ \ts and \ts $I^\bq_{\rqq}$ \ts
denote corresponding ideals in \ts $\cF_n\ts$.  Similarly, let
\ts $\cA_n^\bq=\cF_n/I^\bq_{\cf}$ \ts and \ts $\cB_n^\bq=\cF_n/I^\bq_{\rqq}$ \ts
denote the corresponding \defn{$\bq$-CF} \ts and \defn{$\bq$-RQ algebras}.

\smallskip

\subsection{Inversions and $\bq$-Cayley determinant} \label{ss:def-inv}
Let \. $a_{\lambda_1\mu_1} \. a_{\lambda_2\mu_2}\.\cdots$ \.
be a word in \ts $\cF_n\ts$.
We abbreviate this product as \ts $a_{\lambda,\mu}$ \ts, where
\ts $\lambda=\lambda_1\la_2 \cdots$  \ts and \ts
$\mu=\mu_1\mu_2\cdots$ \ts are words in the alphabet~$[n]$.
For a word $\nu=\nu_1\nu_2 \cdots \in [n]^\ast$, define the \defn{set of inversions}
$$\INV(\nu) \, = \, \big\{(\nu_i,\nu_j) \, : \,  i < j, \, \nu_i > \nu_j\big\},
$$
and let \ts $\inv(\nu) := |\INV(\nu)|$ \ts be the \defn{number of inversions}.
When \ts $\nu$ \ts is a permutation in~$S_n\ts$, this is the usual number of
inversions.

For all \ts $J\subseteq \stirlingI{n}{2}\ts,$ denote
$$
 \bq^{J} \, := \, \prod_{(i,j)\in J} \. q_{ji} \quad \text{and} \quad
 \bq^{-J} \, := \. \big(\bq^J\big)^{-1}.
$$
Define the \ts \defn{$\bq$-Cayley determinant} \ts by
\begin{equation}\label{eq:bq-Cdet}
\Cdet_{\bq}(A)
\, := \, \sum_{\sigma \in S_n} \, (-1)^{\inv(\si)} \bq^{-\INV(\sigma)} \,
 a_{1 \sigma_1} \cdots a_{n \sigma_n}\..
\end{equation}
Here the factors of each term are listed in \emph{column order}, the column
index running $1, \ldots, n$ from left to right.   Thus, \ts $\Cdet_\bq$ \ts is a
also called \emph{column determinant}, where the order of the factors is
essential since variable \ts $a_{ij}$ \ts do not commute.
Note also that \ts $|\INV(\si)|=\inv(\si)$ \ts and \ts $\Cdet_\bq = \Cdet_q$ \ts
when all \ts $q_{ij}\gets q$.

\smallskip
We will need the following algebraic property of the $\bq$-Cayley determinant,
generalizing the standard property for commutative determinants.

\begin{lemma}\label{lem:equal-columns}
Let \ts $A = (a_{ij})\in \cB_n$ \ts be a RQ matrix with two equal columns.
Then \ts $\Cdet(A)=0$.  More generally, let \ts $A = (a_{ij})\in \cB^\bq_n$ \ts
be a $\bq$-RQ matrix with two equal columns.  Then \ts $\Cdet(A)=0$.
\end{lemma}

Versions of this results are given across the literature.
For $q$-RQ matrices, see e.g.\ \cite[Cor.~1(1)]{CFR09}, \cite[Prop.~3.3(3)]{CFRS14}, \cite[$\S$2]{FH07a},
and \cite[Lemma~2.5]{GLZ06}.  For general $\bq$-RQ matrices, see \cite[Lemma~12.2(2) and~12.4]{KP07}
and \cite[Cor.~3.8]{Sil21}.  The proofs are straightforward.


\smallskip

\subsection{Cycle decompositions} \label{ss:def-cycle}
We write every permutation \ts $\si\in S_n$ \ts as a product of disjoint cycles.
We obtain a canonical form of a cycle decomposition \ts $\bC=\bC(\si)$ \ts by
letting each cycle start at its smallest element, which is called the
head of the cycle, and by rearranging the cycles in the increasing
order of their heads.

Formally, a \defn{cycle} \. $C=(c_1, c_2, \ldots, c_i)$ \ts
of length~$i$ is a sequence of \ts $i \ge 1$ \ts
distinct elements in~$[n]$, where the \defn{head} \ts $c_1=\head(C)$ \ts is the
smallest element in~$C$.  The \defn{weight} of~$C$
is the product
\begin{equation}\label{eq:weight-C}
 \weight(C) \, := \, a_{c_1 c_2} \ts a_{c_2 c_3} \cdots a_{c_i c_1}\ts,
\end{equation}
which we view as the product of the weights of \defn{arcs} (directed edges)
in the cycle.  Note that the ordering is crucial here, as the variables
are noncommutative.

A \defn{cycle decomposition} \. $\bC=\bC(\si)$ \ts of a permutation \ts $\si\in S_n$ \ts
is a sequence \ts $\bC=(C_1, \ldots, C_k)$ \ts of disjoint cycles with increasing heads:
\begin{equation} \label{eq:clow-seq}
\head(C_1) \, < \, \ldots \, < \, \head(C_k)\ts.
\end{equation}
Let \ts $k(\bC)$ \ts denote the number of cycles in~$\bC$.
The \defn{weight} \ts and \ts \defn{sign} \ts of \ts $\bC$ \ts are defined as
\begin{equation}\label{eq:weight-product}
\weight(\bC) \, := \, \weight(C_1) \. \cdots\. \weight(C_k)\,, \qquad
 \sign(\bC) \, := \, \sign(\sigma) \, = \, (-1)^{n-k}.
\end{equation}
Denote by \ts $\cS_n=\{\bC(\si) \. : \. \si\in S_n\}$ \ts the set of
cycle decompositions of length~$n$.

For a cycle decomposition \ts $\bC\in \cS_n$ \ts, denote by \ts $\la^\bC$ \ts and \ts $\mu^\bC$ \ts
the corresponding words in \ts $[n]^n$.  We will need the following simple observation.
\smallskip

\begin{lemma}\label{l:q-inv} \.
$\inv\nts\big(\mu^\bCr\big)-\inv\nts\big(\la^\bCr\big) \. = \. n-k(\bCr)$.
\end{lemma}

\smallskip

The proof is by induction.  We omit the details. 

\smallskip

%
\begin{ex}
For a permutation \. $\si = (2,5,1,8,3,6,4,7)\in S_8$\ts, the corresponding
cycle decomposition \.
$\bC= (C_1,C_2,C_3)$, where
\ts $C_1=(1, 2, 5, 3)$, \ts $C_2=(4, 8, 7)$, \ts and $C_3=(6)$. In this case, we have: \.
$\weight(\bC) = a_{12} \. a_{25} \. a_{53} \. a_{31} \, a_{48} \. a_{87} \. a_{74} \. a_{66}$
 \. and \. $\sign(\bC) = -1$.  In the notation above, we have \ts
 $\la^\bC \ts = \ts 1 \ts 2 \ts 5 \ts 3 \ts 4 \ts 8 \ts 7 \ts 6$ \ts
and \ts $\mu^\bC \ts = \ts 2 \ts 5 \ts 3 \ts 1 \ts 8 \ts 7 \ts 4 \ts 6$.  Then \ts
$\inv(\la^\bC) = 5$ \ts and \ts $\inv\nts(\mu^\bC) = 10$.  Since \ts $n=8$ \ts and \ts
$k(\bC)=3$, lemma states in this case that \ts $8-3= 10-5$.
\end{ex}

\smallskip

\subsection{$\bq$-weights and Moore's determinants}\label{ss:def-Moore}
For an \ts $n\times n$ \ts matrix
\ts $A=(a_{ij})$, define \defn{Moore's determinant} \ts as
\begin{equation}\label{eq:Mdet}
\Mdet(A) \, := \, \sum_{\bC\in \cS_n} \. \sign(\bC) \, \weight\nts\big(\bC\big)\ts.
\end{equation}
Note that for \ts $A\in \cA_{\cf}\ts$, we trivially have \. $\Mdet(A)=\Cdet(A)$.

More generally, define the \defn{$q$-weight} \ts as
\begin{equation}\label{eq:q-weight}
\aligned
\weight_q{\nts}(\bC) \, :&=  \, q^{\inv(\la^\bC)-\inv(\mu^\bC)} \, \weight{\nts}(\bC) \\
& = \, q^{k(\bC)-n} \ts \weight{\nts}(\bC)\ts,
\endaligned
\end{equation}
where the second equality holds by Lemma~\ref{l:q-inv}.
Now define the \defn{$q$-Moore determinant} \ts as
\begin{equation}\label{eq:q-Mdet}
\Mdet_q(A) \, := \, \sum_{\bC\in \cS_n} \. \sign(\bC) \. \weight_q{\nts}(\bC)\ts.
\end{equation}

Similarly, define the \defn{$\bq$-weight} \ts as
\begin{equation}\label{eq:bq-weight}
\aligned
\weight_\bq{\nts}(\bC) \, :&=  \, \bq^{\INV(\la^\bC)} \. \bq^{-\INV(\mu^\bC)} \, \weight{\nts}(\bC)\ts,
\endaligned
\end{equation}
and note that $\bq$-weight is equal to the $q$-weight when all \ts $q_{ij}=q$.
Finally, define the \defn{$\bq$-Moore determinant} \ts as
\begin{equation}\label{eq:bq-Mdet}
\Mdet_\bq(A) \, := \, \sum_{\bC\in \cS_n} \.  \. \sign(\bC) \, \weight_\bq{\nts}\big(\bC\big)\ts.
\end{equation}


\smallskip

\subsection{Clow sequences} \label{ss:clow}

If we relax the requirement that all elements of a cycle are distinct, we
arrive at the concept of a \emph{closed ordered walk}, or \emph{clow} for
short, in the terminology of Mahajan and Vinay
\cite{MV97}.\footnote{Valiant~\cite{Val92}, who introduced this concept,
called this a \emph{$c_1$-loop}.}
Formally, a \defn{clow} \ts is a sequence \ts $C=(c_1, c_2, \ldots, c_i)$
\ts of \. $\length(C)=i \ge 1$, such that the head~$c_1$ of~$C$,
denoted \ts $\head(C)=c_1$, is the unique smallest element: \.
$c_1 < \min\{c_2, \ldots, c_i\}$.
The weight of~$C$ is again defined by \eqref{eq:weight-C}.

A \defn{clow sequence} \. $\bC=(C_1, \ldots, C_k)$ \ts is a sequence
of clows with strictly increasing sequence of heads, as in \eqref{eq:clow-seq}.
Let \ts $k(\bC)=k$ \ts denote the number of clows in~$\bC$.
The \defn{weight} \ts and the \defn{sign} \ts of~$\bC$ are defined by
\eqref{eq:weight-product}.  Similarly, \defn{length} of~$\bC$ is defined
as the sum of lengths of the clows:
$$
\length(\bC) \, := \, \.\length(C_1) \. + \, \ldots \, + \.  \length(C_k).
$$

Denote by \ts $\cC_n$ \ts the set of clow sequences of length~$n$.
By definition, every cycle decomposition as above is a clow sequence, i.e.\
$\cS_n \subseteq \cC_n\ts$.  Denote by \ts
$\cC_n^{\ts\tancirc} := \ts \cC_n \sm \cS_n$ \ts the set of clow
sequences which are not cycle decompositions.
For a clow \ts $\bC\in \cC_n$ \ts, denote by \ts $\la^\bC$ \ts and \ts $\mu^\bC$ \ts
the corresponding words in \ts $[n]^n$. Let \defn{$q$-weight} \ts and
\ts \defn{$\bq$-weight} \ts be defined by \eqref{eq:q-weight} and
\eqref{eq:bq-weight}, respectively.

\smallskip

\subsection{Valiant's determinant}\label{ss:def-Val}
Let \ts $A=(a_{ij})$ \ts be an \ts $n\times n$ \ts matrix in~$\cF_n\ts.$
By analogy with Moore's determinant, define \defn{Valiant's determinant} \ts as
\begin{equation}\label{eq:Vdet}
\Vdet(A) \, := \, \sum_{\bC\in \cC_n} \. \sign(\bC) \, \weight\nts\big(\bC\big)\ts,
\end{equation}
where the only difference with \eqref{eq:Mdet} is the summation over a larger
set of clow sequences.

Similarly, define the \defn{$q$-Valiant determinant} \ts as
\begin{equation}\label{eq:q-Vdet}
\Vdet_q(A) \, := \, \sum_{\bC\in \cC_n} \. \sign(\bC) \. \weight_q{\nts}(\bC)\ts,
\end{equation}
and the \defn{$\bq$-Valiant determinant} \ts as
\begin{equation}\label{eq:bq-Vdet}
\Vdet_\bq(A) \, := \, \sum_{\bC\in \cC_n} \.  \. \sign(\bC) \, \weight_\bq{\nts}\big(\bC\big)\ts.
\end{equation}
By design, except for the summation running over a larger set of clow sequences,
these coincide with deformations of Moore determinants given in \eqref{eq:q-Mdet}
and \eqref{eq:bq-Mdet}, respectively.

\smallskip

\subsection{Noncommutative permanents}\label{ss:def-per}
Let \ts $A=(a_{ij})$ \ts be an \ts $n\times n$ \ts matrix in~$\cF_n\ts.$
The \defn{ARQ matrices} \ts mentioned in the introduction are defined as
\begin{equation} \label{eq:ARQ}
\left\{\,\aligned
 a_{kj}\. a_{ki} \ \  &=&& \, -\ts  a_{ki}\. a_{kj}\,, \ \ &&\text{for}  \ \quad i < j\ts,\\
 a_{kj}\. a_{\ell i} \, + \, a_{ki}\. a_{\ell j} \ \  &=&& \, \ts a_{\ell i}\. a_{kj}
 \, + \, a_{\ell j}\. a_{ki}
 \ \ &&\text{for}  \ \quad i<j, \, k<\ell\ts.
 \endaligned\right.
\end{equation}
We use \ts $\cB^-_n$ \ts denote the algebra of these matrices, and note that \ts
$\cB^{-}_n =\cB^{q}_n$ \ts for $q=-1$.  This algebra goes back to Manin \cite{Man89},
and represents a typical example of quantum algebras whose properties change
at roots of unity, see e.g.\ \cite{Lus90,GK93}.

\smallskip

For a matrix \ts $A=(a_{ij})\in \cF_n$ \ts, we can now define a
(noncommutative) \defn{Cayley permanent}:
\begin{equation}\label{eq:Cper}
\Cper(A) \, := \, \sum_{\sigma \in S_n} \, a_{1 \sigma_1} \cdots a_{n \sigma_n}\.,
\end{equation}
Similarly, define (noncommutative) \defn{Moore} \ts and \defn{Valiant permanents}:
\begin{equation}\label{eq:MVper}
\Mper(A) \, := \, \sum_{\bC\in \cS_n} \, \weight{\nts}(\bC)\ts, \qquad
\Vper(A) \, := \, \sum_{\bC\in \cC_n} \, \weight{\nts}(\bC)\ts.
\end{equation}
Since \ts $\sign(\bC) = (-1)^{\inv(\si)} = (-1)^{n-k(\bC)}$ \ts for $\bC=\bC(\si)$,
substituting \ts $q=-1$ \ts in definitions \ts \eqref{eq:Cdet}, \eqref{eq:qCdet} and \eqref{eq:bq-Cdet},
gives:
\begin{equation}\label{eq:det-per}
\Cper(A) \, = \, \Cdet_{-1}(A)\ts, \qquad
\Mper(A) \, = \, \Mdet_{-1}(A)\ts, \qquad
\Vper(A) \, = \, \Vdet_{-1}(A)\ts.
\end{equation}

\medskip


\section{Proof outline} \label{s:outline}

The elaborate setup we presented allows us to reduce Main Theorem~\ref{t:main}
to two results: one combinatorial theorem and one algorithmic lemma.

\smallskip

\subsection{Equality of determinants} \label{ss:outline-det}
Given the different nature of three families of noncommutative
determinants we consider, the following result is quite unexpected.

\smallskip

\begin{thm}[{\rm \defna{Determinantal equality}\ts}{}]\label{t:det-eq}
Let \. $A=(a_{ij}) \in \cB_n^\bq$ \. be a $\bq$-RQ matrix.  Then we have:
\begin{equation}\label{eq:thm-bq-det}
\Cdet_\bq(A) \ = \ \Mdet_\bq(A) \ = \ \Vdet_\bq(A)\ts.
\end{equation}
In particular, for \ts $A\in \cB_n$ \ts we have:
\begin{equation}\label{eq:thm-det}
\Cdet(A) \ = \ \Mdet(A) \ = \ \Vdet(A)\ts.
\end{equation}
Similarly, for \ts $A\in \cB_n^-\ts$ we have:
\begin{equation}\label{eq:thm-per}
\Cper(A) \ = \ \Mper(A) \ = \ \Vper(A)\ts.
\end{equation}
\end{thm}

\smallskip

Note that in \eqref{eq:thm-bq-det}, both equalities are special for
$\bq$-right-quantum matrices and do not hold in greater generality.
Indeed, the Cayley and Moore determinants are summations with equal
number of product terms, but the terms are different, so the equality
fails in \ts $\cF_n\ts$ for \ts $n\ge 2$.  Worse, the summations in the Moore and Valiant
determinants have different number of terms, so the equality fails in
\ts $\cF_n$ \ts again.
As before, \eqref{eq:thm-det} and \eqref{eq:thm-per} follow
from \eqref{eq:thm-bq-det} and \eqref{eq:det-per},
by setting all \ts $q_{ij}=1$ \ts and \ts
$q_{ij}=-1$, respectively.

\smallskip

\subsection{Two combinatorial lemmas}\label{ss:outline-lemmas}
We prove Theorem~\ref{t:det-eq} in two separate lemmas,
each by a combinatorial argument.

\smallskip

\begin{lemma}[{\rm =Theorem~\ref{t:cyc-rqij}}{}]\label{lem:C-M-det}
$\Cdet_\bq(A) = \Mdet_\bq(A)$ \. for all \. $A \in \cB_n^\bq$\ts.
\end{lemma}

\smallskip

We prove the lemma by an explicit many-to-many combinatorial
argument generalizing bijections.  Our argument is somewhat
similar and inspired by the Konvalinka--Pak proof of the $\bq$-MMT
mentioned in~$\S$\ref{ss:back-quantum}.

The idea of the proof is as follows.  It is easier to consider RQ case first.
Both determinants are sums of products over permutations \ts $\si \in S_n$\ts,
one in natural order, and one in cyclic decomposition order.  Start
rearranging products in natural order to products in cycle
decomposition order by swapping adjacent letters, one at a time.

During the swaps, by the commutation relations \eqref{eq:rq}
products will either remain (when the letters commute),
or change (in case of the second relation).  Despite the change,
we can group products in pairs, and establish a bijection at every
swap.  Iterating the procedure gives the desired many-to-many
map proving the lemma.

In the $\bq$-RQ case, additional products of constants \ts $q_{ij}$ \ts
will emerge under every swap under relations in \eqref{eq:bq-RQ}.
Fortunately, products of these constants give the desired $\bq$-weights
as in \eqref{eq:bq-weight}, regardless of the order of swaps.

\smallskip

\begin{lemma}[{\rm =Theorem~\ref{t:clow-rqij}}{}]\label{lem:M-V-det}
$\Mdet_\bq(A) = \Vdet_\bq(A)$ \. for all \. $A \in \cB_n^\bq$\ts.
\end{lemma}

\smallskip

We prove the lemma by an explicit involution which cancels the
terms on the right, corresponding to clows which are \emph{not} \ts
cycle decompositions, thus leaving only the terms on the left.
Our argument emulates the explicit involution introduced
by Mahajan--Vinay \cite{MV97}, but with many complications
added by the noncommutative setting.

The idea of the Mahajan--Vinay involution is roughly as follows.
For a clow of length $n$ that is not a cycle decomposition,
find the ``first violation'', which in this case is a repetition
of a label in the same clow or different clows.  Setting aside the
precise meaning of the ``first violation'', which are different
in \cite{MV97} and this paper anyway, let's say the
repetition occurs in the same clow.  Extract the portion of
the clow between repeated labels, and let it form a new clow
with the head being the repeated label.
When the repetition occurs in the different clows, employ
the opposite procedure.  Since the parity of the number of
clows changes, the corresponding signs in \eqref{eq:Vdet}
are the opposite, allowing the cancellation.

Now, in the RQ case, the process of ``extracting'' a portion
of the clow is complicated by the fact that it has to be moved
across several other clows before it can be placed into
its place in the clow sequence whose heads are increasing
as in \eqref{eq:clow-seq}.  This requires a large number of
swaps which need to be made according to relations \eqref{eq:rq},
so a naive approach immediately falls apart.

We make several major changes in this approach, starting with the
notions of the ``first violation'', and modifying the two-for-two
swapping discussed above into a delicate process involving
cancellations of terms.  At the end, we reduce the problem to
computing Moore's determinant or a matrix with two equal
columns.  By Lemma~\ref{lem:C-M-det}, this reduces to Cayley's
determinant of a matrix with equal columns, which is known to
be zero by Lemma~\ref{lem:equal-columns}.

In the $\bq$-RQ case, the whole process is further complicated by
the need to involve bookkeeping of the products of constants \ts
$q_{ij}\ts$, which we ultimately do by mildly modifying the case
of RQ matrices.

\smallskip

\subsection{Proof of Main Theorem~\ref{t:main}}\label{ss:outline-DP}
The rest of the proof, at least for \ts $\bq=1$, is very close to that
in Mahajan--Vinay argument by design.

\smallskip

\begin{lemma}[{\rm \defn{Dynamic programming}, see \cite{MV97}}{}]\label{lem:DP}
For all nonzero \ts $\bq=(q_{ij})$ \ts as above, there is an ABP of size
$O(n^3)$, which computes the $\bq$-Valiant determinant in free associative
algebra $\cF_n=\cc\<a_{ij} \,:\,  1\le i,j\le n\>$.
In particular, it computes the $\bq$-Valiant determinant of $\bq$-RQ matrices,
the $($usual$)$ Valiant determinant of RQ matrices,
and the Valiant permanent of ARQ matrices.
\end{lemma}

\smallskip

The proof follows Mahajan and Vinay \cite{MV97}, with delicate modifications
to account for the constants \ts $q_{ij}$ \ts computed along the way.  The proof
given in Section~\ref{s:DP}, crucially relies on the Factorization
Lemma~\ref{l:factorization} for the $\bq$-weights of clow sequences.
While the proof of the lemma is very short, this technical result allows
us to use the dynamic programming approach as in the original \ts $\bq=1$ \ts case.

In summary, by Theorem~\ref{t:det-eq}, we obtain an ABP of polynomial size, which
computes the $\bq$-Cayley determinant of $\bq$-RQ matrices.
Consequently, after specializations \ts $q_{ij} = 1$ \ts and \ts
$q_{ij} = -1$, it computes the Cayley determinant and the Cayley permanent,
of right-quantum and antisymmetric right-quantum matrices, respectively.
This completes the proof of Theorem~\ref{t:main}.  \qed

\smallskip

\subsection{Structure of the paper}\label{ss:outline-structure}
We begin with Factorization Lemma in Section~\ref{s:factorization}.
Then, in Section~\ref{s:DP}, we give a $\bq$-version of the Mahajan--Vinay
argument in the proof of Lemma~\ref{lem:DP}.   This is
the only argument which applies to the whole free associative
algebra \ts $\cF_n\ts.$  The reader familiar with the \cite{MV97}
and interested only in the right-quantum ($\bq=1$) case, can safely skip this section.

To get the reader familiar with the setting of words in quantum algebras,
we have Sections~\ref{s:CF} and~\ref{s:RQ}, where we proof Theorem~\ref{t:det-eq}
for the $\bq$-CF and RQ matrices, respectively.   The proof of the $\bq$-RQ case
combines the details of both previous sections, and is given in Section~\ref{s:bq-RQ}.
We conclude with final remarks and open problems in Section~\ref{s:finrem}.

\medskip


\section{Factorization lemma} \label{s:factorization}
In this section we proof that the $\bq$-weight of any clow sequence factors into $\bq$-weights of its clows. This factorization lemma justifies the adjustment of the Mahajan and Vinay~\cite{MV97} dynamic programming by the constants \ts $q_{ij}\ts.$ We start with computation of the $\bq$-weight of a single clow.

\begin{lemma}\label{l:q-weight of clow}
For any clow \ts $C=(c_1, c_2, \ldots, c_\ell)$, we have
$$
\weight_\bq(C) \, = \, a_{c_1 c_2}\ts a_{c_2 c_3} \. \cdots \. a_{c_\ell c_1}\;
 \prod_{i=2}^{\ell} q_{\ts c_1\ts c_i}^{-1} \..
$$
\end{lemma}
\begin{prf}
For a single cycle \ts $C = (c_1, c_2, \ldots, c_\ell)$ \ts with head~$c_1$, we have: $\lambda^{C} = (c_1, \ldots, c_\ell)$ \ts and \ts
$\mu^{C} = (c_2, \ldots, c_\ell, c_1)$, so $\mu^{C}$ is a cyclic shift of
$\lambda^{C}$ by one position. Since $c_1<c_i$ any pair $(c_i,c_1) \in \INV(\mu)$ and any pair $(c_1,c_i) \notin \INV(\lambda)$ for all $2 \leq i \leq \ell$. For all other pairs $(c_i,c_j)$ with $2 \leq i<j \leq \ell$ we have that $(c_i,c_j) \in \INV(\lambda)$ if and only if $(c_i,c_j) \in \INV(\mu)$. Then, by  the definition \eqref{eq:bq-weight}, we have
$$
\weight_\bq(C) \, = \, \bq^{\INV(\la^{C})} \. \bq^{-\INV(\mu^{C})} \, \weight{\nts}(C) \, = \, a_{c_1 c_2}\ts a_{c_2 c_3} \. \cdots \. a_{c_\ell c_1}\;
 \prod_{i=2}^{\ell} q_{\ts c_1\ts c_i}^{-1} \.,
$$
as desired.
\end{prf}

\smallskip
\begin{lemma}[{\rm Factorization Lemma}{}]\label{l:factorization}
For a clow sequence $\bCr=(C_1,  \ldots, C_k) \ts \in \cC_n$, we have:
$$
\weight_\bq(\bCr) \, = \, \weight_\bq(C_1) \. \cdots \. \weight_\bq(C_k) \..
$$
\end{lemma}
\begin{prf}
Since the clows occupy contiguous blocks of positions, we have
$$
\weight(\bC) \, = \, \weight(C_1)\. \cdots \. \weight(C_k) \..
$$
For the product, write $\lambda = \lambda^{C_1}\cdots\lambda^{C_k}$ and
$\mu = \mu^{C_1}\cdots\mu^{C_k}$.
An inversion straddling two distinct clows' blocks pairs the same two letters
in~$\lambda$ as in~$\mu$, since each clow word carries the same multiset of letters;
its contributions to the $\INV(\lambda)$ and $\INV(\mu)$ are therefore equal
and cancel. Thus only within-clow inversions survive, which asserts the claim of the lemma.
\end{prf}

\begin{ex} \label{ex:factorization}
\normalfont
Consider the clow sequence \ts $\bC = (C_1, C_2)$, where \ts
$C_1 = (1\ts 4\ts 2\ts 4)$ \ts and \ts $C_2 = (3\ts 5\ts 4)$.  Here
$$
\weight(\bC) \, = \, a_{14}\ts a_{42}\ts a_{24}\ts a_{41}
 \,\cdot\, a_{35}\ts a_{54}\ts a_{43}\ts,
$$
so that \ts $\la^\bC = 1\ts 4\ts 2\ts 4\ts 3\ts 5\ts 4$ \ts and \ts
$\mu^\bC = 4\ts 2\ts 4\ts 1\ts 5\ts 4\ts 3$.
By Lemma~\ref{l:q-weight of clow} applied to each clow, we have:
$$
\weight_\bq(C_1) \, = \, a_{14}\ts a_{42}\ts a_{24}\ts a_{41}\;
 q_{\ts 1 4}^{-1}\ts q_{\ts 1 2}^{-1}\ts q_{\ts 1 4}^{-1}
\qquad \text{and} \qquad
\weight_\bq(C_2) \, = \, a_{35}\ts a_{54}\ts a_{43}\;
 q_{\ts 3 5}^{-1}\ts q_{\ts 3 4}^{-1}\ts.
$$
By the Factorization Lemma~\ref{l:factorization}, we conclude:
$$
\weight_\bq(\bC) \, = \, q_{\ts 1 2}^{-1}\, q_{\ts 1 4}^{-2}\,
 q_{\ts 3 4}^{-1}\, q_{\ts 3 5}^{-1}\; \weight(\bC)\ts.
$$
Alternatively, computing directly from the definition \eqref{eq:bq-weight}, we have \ts
$\inv(\la^\bC) = 4$ \ts and \ts $\inv(\mu^\bC) = 9$, with
$$
\bq^{\INV(\la^\bC)} \, = \, q_{\ts 2 4}\, q_{\ts 3 4}^{2}\, q_{\ts 4 5}
\qquad \text{and} \qquad
\bq^{-\INV(\mu^\bC)} \, = \, q_{\ts 1 2}^{-1}\, q_{\ts 1 4}^{-2}\,
 q_{\ts 2 4}^{-1}\, q_{\ts 3 4}^{-3}\, q_{\ts 3 5}^{-1}\, q_{\ts 4 5}^{-1}\ts,
$$
giving the same result.
\end{ex}

\medskip

\section{Proof of Dynamic Programming Lemma~\ref{lem:DP}}\label{s:DP}
The proof follows the approach by Mahajan and Vinay~\cite{MV97}  and the exposition
of Rote~\cite{Rote01}, with adjustments for the constants \ts $q_{ij}\ts.$

\smallskip

\begin{prf}[Proof of Lemma~\ref{lem:DP}]
We build up each clow sequence one arc at a time.
Formally, a \defn{partial clow sequence} \ts to be a prefix of the
weight \eqref{eq:weight-product} with heads satisfying \eqref{eq:clow-seq}.
Here we have several (finished) clows trailed by one \defn{open} (incomplete) clow.
For a partial clow sequence \. $\bC = C_1 \cdots C_s\ts
(c_0\ts c_1\ts c_2 \cdots c_t\ts$, define its
\defn{partial ${\bq}$-weight} \ts as
$$
\pweight_{\bq}(\bC) \, := \,
 \weight_\bq(C_1) \. \cdots \. \weight_\bq(C_s) \;\cdot\;
 a_{c_0 c_1}\ts a_{c_1 c_2} \cdots a_{c_{t-1}\ts c_t} \;\cdot\;
 \prod_{j=1}^{t} \. q^{-1}_{\ts c_0\ts c_j}\ts.
$$
Here we take the product of the $\bq$-weights of its completed clows, the weight
$($not the $\bq$-weight$)$ of the open clow, and the scalar correction
of the open clow.
The remainder of the computation depends on is the head and the present
endpoint of the open clow, plus the sign $(-1)^k$ recording the parity
of the count~$k$ of finished clows.

The desired ABP is given by the recursion, where the level of vertices
will be denoted by~$\ell$.   The recursion computes a polynomial
\ts $F_v(a_{ij})$ \ts for every vertex \ts $v=[\ell, c, c_0, s]$,
where \. $1 \le \ell \le n$, \. $1\le c_0< c\le n$, and $s\in \{\pm 1\}$.
The rule is:
\begin{quote}
For $v=[\ell, c, c_0, s]$, let \ts $F_v$ \ts be the sum of partial ${\bq}$-weights
over all partial clow sequences that have reached length~$\ell$,
currently stand at vertex~$c$ inside a clow with head~$c_0$,
and have completed a number of clows of parity \ts $s = \pm 1$.
\end{quote}
Add an auxiliary source vertex $S$ and set \ts $F_S:=1$, and take edges to
all vertices on level~$1$.   For $v=[1,r,r,1]$,
let \ts $F_v := a_{rr}\ts$.  From a given $[\ell, c, c_0, s]$, a partial
clow sequence advances in exactly two ways, which we record as two families of outgoing
edges:
\begin{itemize}
\item \emph{extend} \ts the current clow: edge to \ts $[\ell+1, c', c_0, s]$ \ts of weight \.
$a_{cc'} \cdot q_{\ts c_0\ts c'}^{-1}\ts,$ for all \ts $c' > c_0\ts,$
\item \emph{close} \ts the current clow and open a new one: edges to
\ts $[\ell+1, c_0', c_0', -s]$ \ts of weight~$\ts a_{c c_0}\ts$, \\ for all \ts $c_0' > c_0\ts.$
\end{itemize}
Add a sink vertex \ts $T=[n+1, n+1, n+1, \pm 1]$, with edges from all
vertices at level~$n$.
By the construction and definition \eqref{eq:bq-Vdet},
the sum of weights over all $S\to T$ path, taken with the sign given by~$s$,
equals the $\bq$-Valiant determinant.  Finally, note that the total number
of vertices is \ts $O(n^3)$, which completes the proof of the first part.

For the $\bq$-Valiant determinant and the (usual) Valiant determinant,
the lemma follows from definition and by setting all \ts $q_{ij}=1\ts,$
respectively.
Finally, for the Valiant permanent, the proof follows by the definition,
from the relation \eqref{eq:det-per},  and by setting all \ts
$q_{ij}=-1\ts$.
\end{prf}

\smallskip

{\small
\begin{rem}
In the proof above, we are using the fact that for a complete clow sequence
\ts $\bC = (C_1, \ldots, C_k) \in \cC_n\ts$, the partial ${\bq}$-weight coincides
with the $\bq$-weight:
$$
\pweight_{\bq}(\bC) \, = \, \weight_\bq(\bC)\ts.
$$
This follows immediately from Lemma~\ref{l:q-weight of clow} and  Factorization
Lemma~\ref{l:factorization}.  However, for general partial clow sequences,
the two are not necessarily equal, and can differ by a scalar (product of constants~$q_{ij}$).
For example, for a partial clow \ts $(1\ts 3\ts 2$ \ts we have:
$$
\weight_\bq\bigl(a_{13}\ts a_{32}\bigr) \, = \, a_{13}\ts a_{32}\; q_{\ts 2 3}^{-1}
\qquad \text{while} \qquad
\pweight_{\bq}\bigl(a_{13}\ts a_{32}\bigr) \, = \, a_{13}\ts a_{32}\; q_{\ts 1 3}^{-1}\ts q_{\ts 1 2}^{-1}\ts,
$$
while for a complete clow \ts $(1\ts 3\ts 2)$ \ts we have:
$$
\pweight_{\bq}\bigl(a_{13}\ts a_{32}\ts a_{21}\bigr) \, = \, a_{13}\ts a_{32}\ts a_{21}\; q_{\ts 1 3}^{-1}\ts q_{\ts 1 2}^{-1}
\, = \, \weight_\bq\bigl(a_{13}\ts a_{32}\ts a_{21}\bigr)\ts.
$$
Note that the partial ${\bq}$-weight can be computed recursively, as at
each step the scalar is multiplied by \ts $q^{-1}_{\ts c_0\ts c'}\ts$,
while the usual $\bq$-weight might gain a few new inversions not involving
the head~$c_0$, and determining them requires a further search.
\end{rem}
}
\medskip


\section{Cartier--Foata matrices} \label{s:CF}

\subsection{Cayley--Moore determinantal equality} \label{ss:CF-CM-det}
In some sense, Cayley determinant \eqref{eq:bq-Cdet} and Moore determinant \eqref{eq:bq-Mdet}
are the most economical formulas for the \ts $\bq$-determinant.  In the $\bq$-CF algebra,
they are equal up to a constant.

We start with a general observation which will be used repeatedly.
For a word \. $a_{\la,\mu} = a_{\la_1\mu_1}\ts a_{\la_2\mu_2} \cdots$ \.
in \ts $\cF_n\ts$, define its
\defn{$\bq$-weight} \ts by analogy with \eqref{eq:bq-weight}:
\begin{equation}\label{eq:bq-weight-word}
\weight_\bq\big(a_{\la,\mu}\big) \, := \,
 \bq^{\INV(\la)} \. \bq^{-\INV(\mu)} \, a_{\la,\mu}\ts,
\end{equation}
so that \ts $\weight_\bq(\bC) = \weight_\bq\big(a_{\la^\bC\nts,\ts\mu^\bC}\big)$ \ts
for all clow sequences \ts $\bC\in\cC_n\ts$.

\smallskip

\begin{lemma}\label{l:interchange}
Let \ts $a_{\la',\mu'}$ \ts be obtained from \ts $a_{\la,\mu}$ \ts by
interchanging two adjacent factors \ts $a_{\ell j}\. a_{ki}$ \ts with \ts
$i \ne j$.  Then:
$$
\weight_\bq\big(a_{\la',\mu'}\big) \, = \, \weight_\bq\big(a_{\la,\mu}\big)
\quad \text{in the \ts $\bq$-CF algebra} \ \ \cA_n^\bq\ts.
$$
\end{lemma}

\begin{prf}
Without loss of generality, let \ts $i<j$.  The interchange changes the
inversion sets of~$\la$ and~$\mu$ only in the slot these factors occupy: it
deletes the inversion of~$\mu$ there, so the $\bq$-weight loses the factor
$q_{ij}^{-1}$; in~$\la$ it deletes the inversion contributing $q_{k\ell}$ when
$k<\ell$, adds the inversion contributing $q_{\ell k}$ when $k>\ell$, and
changes nothing when $k=\ell$.  Hence the interchange multiplies the
$\bq$-weight by
$$
 c \, = \, q_{ij}\ts q_{k\ell}^{-1}\ \ (k<\ell), \qquad
 c \, = \, q_{ij}\ts q_{\ell k}\ \ (k>\ell), \qquad
 c \, = \, q_{ij}\ \ (k=\ell).
$$
On the other hand, in each of these three cases relations \eqref{eq:bq-CF}
read \ts $a_{\ell j}\ts a_{ki} \ts = \ts c\ts a_{ki}\ts a_{\ell j}\ts$, so the
interchange multiplies the product of entries by \ts $c^{-1}$.  The two
factors \ts $c$ \ts and \ts $c^{-1}$ \ts cancel, and the $\bq$-weight is
unchanged.
\end{prf}

\smallskip

Now, comparing \eqref{eq:bq-Cdet} and \eqref{eq:bq-Mdet},
we need to show that

\begin{lemma}\label{l:clow-cancel}
For all \ts $\si\in S_n$ \ts and the corresponding cycle decomposition \ts
$\bCr=\bCr(\si)$, we have
$$
\weight_\bq(\bCr) \, = \, \bq^{-\INV(\sigma)} \,
 a_{1 \sigma_1} \cdots a_{n \sigma_n}\..
$$
\end{lemma}
\begin{prf}
Both \ts $\weight(\bCr)$ \ts and the column word \ts
$a_{1 \sigma_1} \cdots a_{n \sigma_n}$ \ts are products of the same $n$
factors, with distinct row indices and distinct column indices, so one is
obtained from the other by interchanging adjacent factors \ts
$a_{\ell j}\. a_{ki}$ \ts with \ts $i \ne j$.  By Lemma~\ref{l:interchange},
we have \ts
$\weight_\bq(\bCr) = \weight_\bq\big(a_{1 \sigma_1} \cdots a_{n \sigma_n}\big)$
\ts in $\cA_n^\bq\ts$.  For the column word we have \ts $\la = (1, \ldots, n)$ \ts
and \ts $\mu = \sigma$, so \ts $\bq^{\INV(\la)} = 1$ \ts and \ts
$\bq^{-\INV(\mu)} = \bq^{-\INV(\sigma)}$, and the result follows
from~\eqref{eq:bq-weight-word}.
\end{prf}

\smallskip

\begin{cor} \label{c:C-M-cf}
Let \. $A=(a_{ij}) \in \cA_n^{\bq}$ \. be a $\bq$-CF matrix.  Then we have:
$$
\Cdet_\bq(A) \, = \, \Mdet_\bq(A).
$$
\end{cor}

\begin{prf}
The result follows by comparing the definitions \eqref{eq:bq-Cdet}
and \eqref{eq:bq-Mdet} term by term, via Lemma~\ref{l:clow-cancel}.
\end{prf}

\smallskip

\subsection{Moore--Valiant determinantal equality} \label{ss:CF-MV-det}%
Let us prove Lemma~\ref{lem:M-V-det} for $\bq$-CF matrices.

\smallskip

\begin{lemma} \label{t:clow-cf-qij}
Let \. $A=(a_{ij}) \in \cA_n^{\cf}$ \. be a $\bq$-CF matrix.  Then we have:
$$
\Mdet_\bq(A) \, = \, \Vdet_\bq(A).
$$
\end{lemma}

\begin{prf}
Comparing the definitions \eqref{eq:bq-Mdet} and \eqref{eq:bq-Vdet}, it
suffices to show that the signed $\bq$-weights of the clow sequences in \ts
$\cC_n^{\ts\tancirc}=\cC_n\sm \cS_n$ \ts cancel each other.  We pair off all such clow
sequences, as follows.

Let \ts $\bC = (C_1, \ldots, C_k) \in \cC_n^{\ts\tancirc}\ts$.  We successively
add the clows \ts $C_k, C_{k-1}, \ldots$ \ts until a repetition occurs:
assume that \ts $C_{p+1}, \ldots, C_k$ \ts is a set of disjoint cycles, but
\ts $C_p, C_{p+1}, \ldots, C_k$ \ts contains a repeated element, where \ts
$1 \le p \le k$.  Let \ts $C_p = (c_1, \ldots, c_i)$, and let \ts $c_j$ \ts
be the first element of this clow, moving from the right end to the left
towards the clow head, which is either \ts (A)~equal to an
element~$c_\ell$ in the tail of the clow \ts $(j < \ell \leq i)$, or
(B)~equal to an element of one of the cycles \ts $C_{p+1}, \ldots, C_k\ts$.

Precisely one of these two cases occurs: we call \ts $\bC$ \ts
\defn{merged} \ts in case~(A), and \defn{split} \ts in case~(B).
In case~(A), we remove the cycle \ts
$(c_j = c_\ell, c_{j+1}, \ldots, c_{\ell-1})$ \ts from the clow~$C_p$ and
turn it into a separate cycle, which we insert into~$\bC$ in the position
prescribed by \eqref{eq:clow-seq}, after cyclically shifting its head to the
front.  In case~(B), we insert the corresponding cycle into the clow~$C_p$
just before the element~$c_j\ts$.  These two operations are inverse to each
other, and hence they pair off all clow sequences in \ts
$\cC_n^{\ts\tancirc}$ \ts into \defn{orbits} \ts $\{\bC, \bC'\}$ \ts of a merged
and the corresponding split clow sequence, with \ts $k(\bC') = k(\bC)+1$;
see Example~\ref{ex:merged-split} below.

By \eqref{eq:weight-product}, we have \ts $\sign(\bC') = -\sign(\bC)$.  On
the other hand, one word is obtained from the other by
repeatedly interchanging two adjacent factors \ts $a_{\ell j}\. a_{ki}$ \ts
with \ts $i \ne j$.  By Lemma~\ref{l:interchange}, we conclude:
$$
\weight_\bq(\bC) \, = \, \weight_\bq(\bC')
\quad \text{in the algebra} \ \ \cA_n^\bq\..
$$
Therefore, the two terms of every orbit in \eqref{eq:bq-Vdet} cancel each
other, and only the terms corresponding to cycle decompositions \ts
$\bC\in\cS_n$ \ts survive, giving \eqref{eq:bq-Mdet}, as desired.
\end{prf}

\smallskip
{\small
\begin{rem}\label{rem:MV-diff}
\normalfont
Our pairing between merged and split clow sequences differs from the one of
Mahajan and Vinay~\cite{MV97}.  The key difference is that we read each clow from
its right end towards its head on the left, whereas in~\cite{MV97} each clow is
read from its head towards the right.  This change is crucial: in the
non-commutative classes of matrices we consider, entries in the same row do not
commute at all, and since the heads of the split clow sequence must increase
from left to right, the given involution assures that we do not encounter
factors from the same row as we perform swaps to reach its split sequence.
This was not the case in the original Mahajan--Vinay involution.
\end{rem}
}

\smallskip
\begin{ex} \label{ex:merged-split}
\normalfont
Let \ts $\bC = (C_1, C_2, C_3)$ \ts be the merged clow sequence, where \ts
$C_1 = (1\ts 7\ts 2\ts 5\ts 5\ts 3\ts 8\ts 5\ts 7)$, \ts $C_2 = (2)$ \ts
and \ts $C_3 = (4\ts 6\ts 9)$.  Extracting the cycle \ts $(3\ts 8\ts 5)$ \ts
from the clow~$C_1$ gives the corresponding split clow sequence
$$
 \bC' \, = \, (1\ts 7\ts 2\ts 5\ts 5\ts 7)\ts(2)\ts(3\ts 8\ts 5)\ts(4\ts 6\ts 9)\ts,
$$
so that \ts $\{\bC, \bC'\}$ \ts is an orbit.  In this case, we have:
$$
 \weight(\bC) \, = \, a_{17}\ts a_{72}\ts a_{25}\ts a_{55}\ts \underline{a_{53}}\ts \underline{a_{38}}\ts \underline{a_{85}}\ts a_{57}\ts a_{71}
 \,\cdot\, a_{22} \,\cdot\, a_{46}\ts a_{69}\ts a_{94}
$$
and
$$
 \weight(\bC') \, = \, a_{17}\ts a_{72}\ts a_{25}\ts a_{55}\ts a_{57}\ts a_{71}
 \,\cdot\, a_{22} \,\cdot\, \underline{a_{38}}\ts \underline{a_{85}}\ts \underline{a_{53}} \,\cdot\, a_{46}\ts a_{69}\ts a_{94}\ts,
$$
where the letters of the relocated cycle are underlined.  Note that \ts
$k(\bC) = 3$ \ts and \ts $k(\bC') = 4$, so that \ts
$\sign(\bC') = -\sign(\bC)$, in agreement with \eqref{eq:weight-product}.
By Lemma~\ref{l:q-weight of clow} and the Factorization
Lemma~\ref{l:factorization}, we also have:
$$
\weight_\bq(\bC) \, = \,
 q_{\ts 1 2}^{-1}\, q_{\ts 1 3}^{-1}\, q_{\ts 1 5}^{-3}\,
 q_{\ts 1 7}^{-2}\, q_{\ts 1 8}^{-1}\, q_{\ts 4 6}^{-1}\, q_{\ts 4 9}^{-1}
 \; \weight(\bC)
$$
and
$$
\weight_\bq(\bC') \, = \,
 q_{\ts 1 2}^{-1}\, q_{\ts 1 5}^{-2}\, q_{\ts 1 7}^{-2}\,
 q_{\ts 3 5}^{-1}\, q_{\ts 3 8}^{-1}\, q_{\ts 4 6}^{-1}\, q_{\ts 4 9}^{-1}
 \; \weight(\bC')\ts,
$$
which are equal in the algebra \ts $\cA_n^\bq\ts$, as in the proof of
Lemma~\ref{t:clow-cf-qij}.
\end{ex}


\medskip


\section{Right-quantum matrices} \label{s:RQ}

Unlike in the Cartier--Foata case \ts $\cA_n^\bq\ts$, showing that \ts
$\Cdet_\bq\ts$, \ts $\Mdet_\bq$ \ts and \ts $\Vdet_\bq$ \ts are equal in \ts
$\cB_n^\bq$ \ts is significantly harder.  We postpone the proof until the
next section. 
In this section, we start with the RQ case,
which captures the essence of the proof for the right-quantum matrices.

\smallskip

\subsection{Cayley--Moore determinantal equality} \label{ss:cyc-rq}
In the RQ case, we have the following special case of Lemma~\ref{lem:C-M-det}:

\smallskip

\begin{lemma} \label{t:cyc-rq}
Let \. $A=(a_{ij}) \in \cB_n$ \. be an RQ matrix.  Then we have:
$$
\Cdet(A) \, = \, \Mdet(A).
$$
\end{lemma}

\smallskip

We prove Lemma~\ref{t:cyc-rq} following the bijective approach of Konvalinka and
the first author~\cite{KP07}.  Consider
lattice steps of the form \ts $(x, i) \to (x+1, j)$ \ts for  \ts$x, i, j \in \Z$, \ts $1 \le i, j \le n$.
We think of~$x$ as drawn along the $x$-axis, increasing from left to right, and refer to~$i$
and~$j$ as the \defn{starting} \ts and \defn{ending height}, respectively.  In keeping with
our weight convention, where the arc from~$i$ to~$j$ carries the entry~$a_{ij}$, we represent
the step \ts $(x, i) \to (x+1, j)$ \ts by the variable~$a_{ij}$, and a finite sequence of steps by a
word in the alphabet \ts $\{a_{ij}\}$.

We will consider sequences consisting of \ts $n$ \ts letters \ts $a_{ij}$ \ts whose sets of starting points and
of ending points each have no repetitions, and are exactly \ts $\{1, \ldots, n\}$.  We call such a
sequence \defn{balanced}.  An \defn{$o$-sequence} is a balanced sequence whose columns are
arranged in increasing order.

For a balanced sequence
$$
(0, i_1) \to (1, j_1),\ (1, i_2) \to (2, j_2),\ \ldots,\ (n-1, i_n) \to (n, j_n),
$$
define the \defn{rank}~$r$ as
$$
 r \, := \, \bigl|\set{(s, t) \ \colon \  i_s > i_t,\ 1 \le s < t \le n}\bigr|.
$$
Clearly, $o$-sequences are exactly the balanced sequences of rank~$0$.

Note that the words in the Cayley determinant \eqref{eq:Cdet} are precisely the
$o$-sequences.  Indeed, the term of~$\sigma$ lists its letters by increasing column,
so its starting points run \ts $1, 2, \ldots, n$ in order, while its ending
points are the distinct values $\ts \sigma_1, \ldots, \sigma_n\ts.$

A balanced sequence is called a \defn{$p$-sequence}, if it is the word
\ts $\weight(\bC)$ \ts of the canonical form of a cycle decomposition
$\ts \bC \in \cS_n\ts$, that is, the clow-ordered weight obtained by reading
each cycle from its head and listing the cycles by increasing heads.  In this language,
the summation in Moore's determinant, is taken over all $p$-sequences.

A balanced sequence is called a \defn{$q$-sequence}, if it starts with a prefix of a
$p$-sequence and is then continued with a suffix of an $o$-sequence, with at most one
exceptional step which might have a starting point larger than those of the following
steps to the right.

\begin{ex} \label{ex:poq}
\normalfont
For $\bC = (1\ts 5\ts 2\ts 7)\ts(4)\ts(3\ts 8\ts 6)$, the three types of sequence are
$$
\begin{aligned}
 p &:\quad a_{15}\ts a_{52}\ts a_{27}\ts a_{71}\ts a_{44}\ts a_{38}\ts a_{86}\ts a_{63}, \\
 o &:\quad a_{15}\ts a_{27}\ts a_{38}\ts a_{44}\ts a_{52}\ts a_{63}\ts a_{71}\ts a_{86}, \\
 q &:\quad a_{15}\ts a_{52}\ts a_{27}\ts a_{38}\ts a_{44}\ts a_{71}\ts a_{63}\ts a_{86},
\end{aligned}
$$
the $q$-sequence having $p$-prefix $a_{15}\ts a_{52}\ts a_{27}$ and a single exceptional
step~$a_{71}$.  Drawing each letter~$a_{ij}$ as a step from height~$i$ to height~$j$, the
three paths are as follows.\end{ex}

\begin{center}
\begin{tabular}{@{}c@{\quad}c@{\quad}c@{}}
\stepseqNstepbold{0/1/5, 1/5/2, 2/2/7, 3/7/1, 4/4/4, 5/3/8, 6/8/6, 7/6/3}{
   \node[font=\itshape] at (4,-2.6) {$p$};} \qquad\qquad
&
\stepseqNbold{0/1/5, 1/2/7, 2/3/8, 3/4/4, 4/5/2, 5/6/3, 6/7/1, 7/8/6}{
   \draw[decorate,decoration={brace,mirror,amplitude=2.5pt}]
       (0.05,-0.8) -- (8,-0.8) node[midway,below,yshift=-1pt,font=\tiny]{columns incr.};
   \node[font=\itshape] at (4,-3.4) {$o$};}{0} \qquad\qquad
&
\stepseqNboldhl{0/1/5, 1/5/2, 2/2/7, 3/3/8, 4/4/4, 5/7/1, 6/6/3, 7/8/6}{
   \draw[decorate,decoration={brace,amplitude=2.5pt}]
       (0.05,9.0) -- (3.0,9.0) node[midway,above,yshift=-1pt,font=\tiny]{$p$-pref.};
   \draw[decorate,decoration={brace,amplitude=2.5pt}]
       (3.05,9.0) -- (8.0,9.0) node[midway,above,yshift=-1pt,font=\tiny]{$o$-suf.};
   \draw[red!70!black,->] (5.5,-2.4)
       node[below,font=\tiny,text=red!70!black,align=center]{exceptional step\\$7\!>\!6$}
       -- (5.5,-0.2);
   \node[font=\itshape] at (3.0,-2.4) {$q$};}{3}{5}
\\
\end{tabular}
\vskip.2cm
\text{{\bf Figure~8.1} \. An example of \ts $p$-, \ts $o$- \ts and \ts $q$-sequences.}
\end{center}

\smallskip

\begin{prf}[Proof of Lemma~\ref{t:cyc-rq}]
We define a map~$\Psi$ on $q$-sequences, following the approach by Konvalinka and
the first author~\cite{KP07}.  Given a $q$-sequence, let its longest prefix that already obeys
the rules of a $p$-sequence consist of~$x$ steps, and let~$i$ be the height at the end of
step~$x$.  The map~$\Psi$ swaps the unique step starting at height~$i$ with its immediate left
neighbor.  This single swap raises the rank by~$1$.

A swap of two steps with distinct ends is not by itself a relation, so we carry it out
together with a second $q$-sequence.  This partner has the same rank and agrees with the
first everywhere except that the ends of the two swapped steps are interchanged.  If the
first sequence corresponds to a permutation~$\sigma$, the partner corresponds to
$\sigma' = \sigma\cdot\tau$, where~$\tau$ is the transposition of the two interchanged
ending points; hence the two carry opposite signs, $\sign\sigma' = -\sign\sigma$.  The
second relation in \eqref{eq:rq} now performs the swap in both sequences at once, and because
their signs are opposite the relation applies with the signs intact:
$\mp\,\alpha \pm \beta = \mp\,\Psi\alpha \pm \Psi\beta$.  Example~\ref{ex:swap} illustrates such a
pair.

When a sequence is already a $p$-sequence, $\Psi$ leaves it unchanged.  Iterating the map
therefore rewrites the Cayley determinant in \ts $\cB_n$ \ts as the Moore
determinant, proving Lemma~\ref{t:cyc-rq}.
\end{prf}

\begin{ex} \label{ex:swap}
\normalfont
Consider the two $q$-sequences $\alpha$ and $\beta$ that differ only in the ends of the two
red steps.  Here $\alpha$ corresponds to $\sigma = (5,7,8,4,2,3,1,6)$ with
$\sign\sigma = -1$, while interchanging the ends of the two red steps merges the fixed
point~$4$ into the long cycle, giving $\sigma' = (5,7,8,1,2,3,4,6)$ with
$\sign\sigma' = +1$; the signs shown are these permutation signs.  Applying~$\Psi$ swaps the
two red steps in each row and leaves the signs unchanged, see Figure~8.2.\end{ex}

\begin{center}
\begin{tabular}{@{}c@{\qquad}c@{}}
\stepseqswapC{0/1/5, 1/5/2, 2/2/7, 3/3/8, 4/4/4, 5/7/1, 6/6/3, 7/8/6}{}{4}{5}
&
\stepseqswapC{0/1/5, 1/5/2, 2/2/7, 3/3/8, 4/4/1, 5/7/4, 6/6/3, 7/8/6}{}{4}{5}\\[2pt]
{\footnotesize$\sign\sigma=-1:\ \ \alpha$} & {\footnotesize$\sign\sigma'=+1:\ \ \beta$}\\[10pt]
\multicolumn{2}{c}{\small$\Big\downarrow\ \Psi$ \ (swap the two red steps in each row;\ signs unchanged)}\\[8pt]
\stepseqswapC{0/1/5, 1/5/2, 2/2/7, 3/3/8, 4/7/1, 5/4/4, 6/6/3, 7/8/6}{}{4}{5}
&
\stepseqswapC{0/1/5, 1/5/2, 2/2/7, 3/3/8, 4/7/4, 5/4/1, 6/6/3, 7/8/6}{}{4}{5}\\[2pt]
{\footnotesize$-\,\Psi\alpha$} & {\footnotesize$+\,\Psi\beta$}\\
\end{tabular}

\vskip.35cm
\text{{\bf Figure~8.2} \.  Map \ts $\Psi$ \ts in Example~\ref{ex:swap}.}

\vskip.35cm
\end{center}

\smallskip
\subsection{Moore--Valiant determinantal equality} \label{ss:clow-rq}
In the proof of the Cayley--Moore equality given above, pairs of $q$-sequences remain
pairs of $q$-sequences after a swap.  On the other hand, for pairs of clow
sequences, a sequence of swaps can end up with a pair of balanced words with
no clear description.  Thus, we prove the Moore--Valiant equality using a
different method, namely induction.

\begin{lemma} \label{t:clow-rq}
Let \. $A=(a_{ij}) \in \cB_n$ \. be an RQ matrix.  Then we have:
$$
\Mdet(A) \, = \, \Vdet(A).
$$
\end{lemma}

\begin{ex} \label{ex:rq-joint}
\normalfont
Take \ts $n = 5$, and consider four clow sequences in \ts $\cC_5\ts$:
$$
 \bC_1 = (1\ts 2\ts 2\ts 3)(5), \qquad
 \bC_1' = (1\ts 2\ts 3)(2)(5), \qquad
 \bC_2 = (1\ts 2\ts 3\ts 2)(5), \qquad
 \bC_2' = (1\ts 2)(2\ts 3)(5),
$$
forming two orbits \ts $\{\bC_1, \bC_1'\}$ \ts and \ts $\{\bC_2, \bC_2'\}$.
In this case, we have:
$$
 \weight(\bC_1) = a_{12}\ts a_{22}\ts a_{23}\ts a_{31}\ts a_{55}\., \qquad
 \weight(\bC_1') = a_{12}\ts a_{23}\ts a_{31}\ts a_{22}\ts a_{55}\.,
$$
$$
 \weight(\bC_2) = a_{12}\ts a_{23}\ts a_{32}\ts a_{21}\ts a_{55}\., \qquad
 \weight(\bC_2') = a_{12}\ts a_{21}\ts a_{23}\ts a_{32}\ts a_{55}\..
$$
Here \ts $\sign(\bC_1) = \sign(\bC_2) = -1$ \ts and \ts
$\sign(\bC_1') = \sign(\bC_2') = +1$, and the column relation
in~\eqref{eq:rq} reduces the two signed orbits to
$$
 \weight(\bC_1') - \weight(\bC_1) \, = \, a_{12}\ts a_{23}\ts
 \bigl(a_{31}\ts a_{22} - a_{22}\ts a_{31}\bigr)\ts a_{55}\ts,
$$
$$
 \weight(\bC_2') - \weight(\bC_2) \, = \, a_{12}\ts a_{23}\ts
 \bigl(a_{21}\ts a_{32} - a_{32}\ts a_{21}\bigr)\ts a_{55}\ts.
$$
Neither is zero in \ts $\cB_5\ts$, since the relations \eqref{eq:rq} force neither
$a_{31}$ to commute with~$a_{22}$ nor $a_{21}$ with~$a_{32}$.  Their sum, however,
vanishes: the crossing relation in~\eqref{eq:rq} with \ts $i=1$, $j=2$, $k=2$,
$\ell=3$ \ts makes the inner
factor $\bigl(a_{31}\ts a_{22} - a_{22}\ts a_{31}\bigr) + \bigl(a_{21}\ts a_{32} -
a_{32}\ts a_{21}\bigr)$ identically zero, so
$$
 [\weight(\bC_1') \. - \. \weight(\bC_1)] \. + \. [\weight(\bC_2')\. - \. \weight(\bC_2)] \, = \, 0 \quad \text{in \ $\cB_5\ts,$}
$$
even though neither orbit cancels separately.
\end{ex}

\subsection{Clow types} \label{ss:collective}

In the CF case of Section~\ref{s:CF}, the proof of
Lemma~\ref{t:clow-cf-qij} shows that \ts
$\weight_\bq(\bC) = \weight_\bq(\bC')$ \ts in \ts $\cA_n^\bq$ \ts for each
orbit \ts $\{\bC, \bC'\}$ \ts separately.  Here
this fails: in Example~\ref{ex:rq-joint}, neither \ts
$\weight(\bC_1') - \weight(\bC_1)$ \ts nor \ts
$\weight(\bC_2') - \weight(\bC_2)$ \ts is zero in~$\cB_5\ts$, and yet their sum is zero.

Since the orbit pairing no longer suffices, we prove the collective cancellation of the
non-permutation clow sequences by a coarser grouping: we partition them into several
types and show that the signed weights within each
type sum to zero in \ts $\cB_n\ts$.

Read the clow sequence from right to left, taking the clows in reverse order and each
clow from its right end toward its head.  Let~$e'$ be the first element of this reading
that repeats one already read, let~$e$ be the earlier occurrence it repeats, and let~$h$
be the head of the clow containing~$e'$.  We call that clow \defn{defected}.

Passing to the corresponding split clow sequence separates~$e$ and~$e'$ into two
distinct clows.  We call the one carrying~$e$, namely the clow with the larger head, the
\defn{right split}, and the one carrying~$e'$ the \defn{left split}.
\smallskip

\begin{ex} \label{ex:defected}
\normalfont
The merged clow sequence
$$
 \bC \, = \, (\ts\underset{h}{\underline{1}}\ts 7\ts 2\ts 5\ts \underset{e'}{\underline{\mathbf{5}}}\ts 3\ts 8\ts \underset{e}{\underline{\mathbf{5}}}\ts 7)\ts(2)\ts(4\ts 6\ts 9)
$$
of Example~\ref{ex:merged-split} is read from right to left as
$9,6,4,\ts 2,\ts 7,\mathbf{5},8,3,\mathbf{5},\ldots$, the repeated pair shown in bold:
the second occurrence is the first repeat $e' = \mathbf{5}$, and the earlier one is $e = \mathbf{5}$.
The clow $(1\ts 7\ts 2\ts 5\ts 5\ts 3\ts 8\ts 5\ts 7)$ is the defected clow of~$\bC$, and
the corresponding split clow sequence is
$$
 \bC' \, = \, (\ts\underset{h}{\underline{1}}\ts 7\ts 2\ts 5\ts \underset{e'}{\underline{\mathbf{5}}}\ts 7)\ts(2)\ts(3\ts 8\ts \underset{e}{\underline{\mathbf{5}}})\ts(4\ts 6\ts 9),
$$
with the same three elements underlined: the right split $(3\ts 8\ts 5)$ carries~$e$ and
the left split $(1\ts 7\ts 2\ts 5\ts 5\ts 7)$ carries~$e'$.
\end{ex}
\smallskip

\begin{Def} \label{def:type1}
\normalfont
An orbit is of \defn{type~$1$} if, in the defected clow of its merged sequence, the
head~$h$ is followed by an element other than~$e'$.
\end{Def}
\smallskip

\begin{ex} \label{ex:type1}
\normalfont
The merged clow sequence
$$
 \bC \, = \, (\ts\underset{h}{\underline{1}}\ts 7\ts 2\ts 5\ts \underset{e'}{\underline{5}}\ts 3\ts 8\ts \underset{e}{\underline{5}}\ts 7)\ts(2)\ts(4\ts 6\ts 9)
$$
of Example~\ref{ex:defected} has the block $7\ts 2\ts 5$ between~$h$ and~$e'$, so the
head~$h$ is followed by~$7 \ne e'$.  Hence this orbit is of type~$1$.
\end{ex}

\smallskip
\begin{Def} \label{def:type2}
\normalfont
An orbit is of \defn{type~$2$} if, in the defected clow of its merged sequence, the
head~$h$ is followed by~$e'$, and the sequence contains another clow whose head is
smaller than~$h$.
\end{Def}
\smallskip

\begin{ex} \label{ex:type2}
\normalfont
In the merged clow sequence
$$
 \bC \, = \, (1\ts 3\ts 2)\ts(\ts\underset{h}{\underline{3}}\ts \underset{e'}{\underline{6}}\ts 4\ts \underset{e}{\underline{6}})\ts(5\ts 2),
$$
the defected clow is $(3\ts 6\ts 4\ts 6)$, whose head~$h = 3$ is immediately followed
by~$e' = 6$.  Moreover the clow $(1\ts 3\ts 2)$ has head $1 < h$, so this orbit is of
type~$2$.  Its corresponding split clow sequence is
$$
 \bC' \, = \, (1\ts 3\ts 2)\ts(\ts\underset{h}{\underline{3}}\ts \underset{e'}{\underline{6}})\ts(4\ts \underset{e}{\underline{6}})\ts(5\ts 2),
$$
with the same three elements underlined.  Here the left split $(3\ts 6)$ is not the
leftmost clow of~$\bC'$: the smaller-head clow $(1\ts 3\ts 2)$ stands to its left.
\end{ex}
\smallskip

\begin{Def} \label{def:type3}
\normalfont
An orbit is of \defn{type~$3$} if it is not of type~$1$
or type~$2$.
\end{Def}
\smallskip

\begin{ex} \label{ex:type3}
\normalfont
In the merged clow sequence
$$
 \bC \, = \, (\ts\underset{h}{\underline{1}}\ts \underset{e'}{\underline{5}}\ts 3\ts \underset{e}{\underline{5}}\ts 2)\ts(4\ts 8)\ts(6\ts 7),
$$
the head~$h$ is followed by~$e'$ and no clow of smaller head precedes the defected clow,
so the orbit is neither of type~$1$ nor of type~$2$.  Hence this orbit is of type~$3$.
\end{ex}
\smallskip

\subsection{Collective cancellation} \label{ss:cancellation}

We now use a double induction, to prove that the signed sums of
non-permutation clow sequences vanish in \ts $\cB_n\ts$, type by type.

{\small
\begin{rem} \label{r:support-interval}
\normalfont
Whenever we wish to prove that a signed sum of clow sequences supported on a fixed
multiset of elements vanishes in \ts $\cB_n\ts$, we may assume without loss of
generality that the support is an interval $\{1, 2, \ldots, k\}$, possibly with
repetitions.  Indeed, if the support has a gap, we shift all larger elements down,
preserving their order, so as to fill it.
\end{rem}}

\smallskip

\begin{indasm} \label{ia:types}
\normalfont
Fix $m \ge 1$, the number of elements in the support multiset. We assume that for any $n\leq m$:
\begin{enumerate}
\item the signed sum of the clow sequences of type~$1$ supported on a fixed multiset with
$n$ elements that share a common prefix
$$
 \underline{C_1\ts C_2 \cdots (h\ts t_1\ts t_2 \cdots t_k\ts e'}\ts \cdots) \cdots
$$
vanishes in \ts $\cB_n\ts$;
\item the signed sum of the clow sequences of type~$2$ supported on a fixed multiset with
$n$ elements that share a common prefix
$$
 \underline{C_1\ts C_2 \cdots (h\ts e'}\ts \cdots) \cdots
$$
vanishes in \ts $\cB_n\ts$;
\item the signed sum of the clow sequences of type~$3$ supported on a fixed multiset with
$n$ elements that share a common prefix
$$
 \underline{(h\ts e'}\ts \cdots) \cdots
$$
vanishes in \ts $\cB_n\ts$.
\end{enumerate}
\end{indasm}

\smallskip

\begin{lemma} \label{l:type1}
For any integer $n \ge 1$, the signed sum of the clow sequences of type~$1$ supported on
a fixed multiset with $n$ elements that share a common prefix
$$
 C_1\ts C_2 \cdots (h\ts t_1\ts t_2 \cdots t_k\ts e' \cdots) \cdots
$$
vanishes in \ts $\cB_n\ts$.
\end{lemma}

\begin{prf}
We argue under Inductive Assumption~\ref{ia:types}.
Deleting the common clows left of the defected clow and the block \ts $t_1 \cdots t_k$ \ts turns
the sum into the signed sum of the type-$3$ sequences of order \ts $n < m$, beginning with \ts
$(\ts h\ts e' \cdots) \cdots$.  This sum is zero in \ts $\cB_n$ \ts by induction.

Note that every term in the sum above begins \. $a_{h e'}\ts a_{e', *}\cdots$,
so factoring out~$\ts a_{h e'}$ \ts shows that the signed sum of the suffixes~$S$
is zero in \ts $\cB_n\ts$.  In the original sum, every term shares the prefix
$$
 \underset{\text{common prefix}}{\underline{\cdots\ts a_{h t_1}\ts a_{t_1 t_2} \cdots a_{t_{k-1}\ts t_k}\ts a_{t_k\ts e'}}}\ts a_{e', *}\cdots,
$$
and factoring it out leaves the same suffixes~$S$.  Multiplying by this common
prefix on the left, we conclude that the original sum vanishes in \ts $\cB_n\ts$.
\end{prf}

\begin{lemma} \label{l:type2}
For any integer $n \ge 1$, the signed sum of the clow sequences of type~$2$ supported on
a fixed multiset with $n$ elements that share a common prefix
$$
 C_1\ts C_2 \cdots (h\ts e' \cdots) \cdots
$$
vanishes in \ts $\cB_n\ts$.
\end{lemma}

\begin{prf}
Completely analogous to the proof of Lemma~\ref{l:type1}.
\end{prf}

\begin{rem} \label{r:type3-support}
\normalfont
A non-permutation clow sequence that is not of type~$1$ or type~$2$ has only one pair of
repeated elements, namely $e'$ and~$e$; all of its remaining elements are distinct.
\end{rem}

\begin{lemma} \label{l:type3}
For any integer $n \ge 1$, the signed sum of the clow sequences of type~$3$ supported on
a fixed multiset with $n$ elements that share a common prefix
$$
 (h\ts e' \cdots) \cdots
$$
vanishes in \ts $\cB_n\ts$.
\end{lemma}

\begin{prf}
We argue under Inductive Assumption~\ref{ia:types}.
By Remark~\ref{r:type3-support}, we may assume that the support is
$\{1, 2, \ldots, k, k, \ldots, n\}$ with the single repeated pair $e' = e = k$,
and head $h = 1$.  Let
$$
 T_n \, = \, \sum_{\bC} \, \sign(\bC) \. \weight(\bC)
$$
be the signed sum of all type-$3$ clow sequences supported on this multiset.

Every such sequence has its defected clow of the form $(h\ts e' \cdots)$, so its weight
begins with the arc $h \to e'$, that is, with the letter $a_{h e'}$.  Writing
$\weight(\bC) = a_{h e'}\, w(\bC)$, where \ts $w(\bC)$ \ts denotes the product
of the remaining \ts $n$ \ts letters of \ts $\weight(\bC)$, we factor this
common letter out:
$$
 T_n \, = \, a_{h e'}\, R_n\,, \qquad R_n \, = \, \sum_{\bC} \, \sign(\bC) \. w(\bC).
$$

We claim that
$$
 R_n \, = \, -\,\Mdet(Q)\ts,
$$
where \ts $Q=(Q_{cr})$ \ts is the \ts $n\times n$ \ts matrix obtained
from~$A$ by replacing its first column with a copy of its $k$-th column:
$$
 Q_{1r} \, = \, a_{kr}\,, \qquad Q_{cr} \, = \, a_{cr} \ \ (2 \le c \le n).
$$
Indeed, each entry of~$Q$ is one of the original entries~$a_{cr}$, and the
columns of~$Q$ form a sub-collection of the columns of~$A$, so the entries
of~$Q$ satisfy the column and crossing relations in~\eqref{eq:rq} and \ts
$Q\in \cB_n$ \ts is again an RQ matrix.  Now, deleting the distinguished
repeated~$k$ that follows the head of the defected clow turns a type-$3$ clow
sequence on \ts $\{1, \ldots, k, k, \ldots, n\}$ \ts into a cycle
decomposition \ts $\bC\in \cS_n\ts$, the split arcs \ts
$h \to k \to (\cdot)$ \ts collapsing to a single arc whose letter belongs to
the first column of~$Q$, and conversely.  This bijection preserves the weights
and reverses the signs, since the number of clows is unchanged while the
length drops by one, which proves the claim.

By Lemma~\ref{t:cyc-rq} applied to~$Q$, we have \ts
$\Mdet(Q) = \Cdet(Q)$ \ts in \ts $\cB_n\ts$.  The first and the $k$-th
columns of~$Q$ are equal, both being the $k$-th column of~$A$, and the Cayley
determinant of an RQ matrix with two equal columns vanishes in \ts
$\cB_n\ts$ by Lemma~\ref{lem:equal-columns}.  Hence
$$
 R_n \, = \, -\,\Mdet(Q) \, = \, -\,\Cdet(Q) \, = \, 0
 \quad \text{in} \ \ \cB_n\ts.
$$
Multiplying by the common left factor~$a_{h e'}$, we conclude that \ts
$T_n = a_{h e'}\, R_n$ \ts vanishes in \ts $\cB_n\ts$, as desired.
\end{prf}

\begin{ex} \label{ex:type3-1223}
\normalfont
We illustrate the proof on the multiset $\{1, 2, 2, 3\}$, so $k = 2$ and $n = 3$.  There
are six type-$3$ clow sequences, all with defected clow $(1\ts 2 \cdots)$, namely
$(1\ts 2)(2)(3)$, $(1\ts 2)(2\ts 3)$, $(1\ts 2\ts 2)(3)$, $(1\ts 2\ts 2\ts 3)$,
$(1\ts 2\ts 3\ts 2)$ and $(1\ts 2\ts 3)(2)$, and \. $T_3$ \. is equal to
$$
  -\,a_{12} a_{21} a_{22} a_{33} + a_{12} a_{21} a_{23} a_{32}
 + a_{12} a_{22} a_{21} a_{33} - a_{12} a_{22} a_{23} a_{31}
 - a_{12} a_{23} a_{32} a_{21} + a_{12} a_{23} a_{31} a_{22}\ts.
$$
Factoring out~$a_{12}$ (the arc $1 \to 2$) leaves
$$
 R_3 \, = \, -\,a_{21} a_{22} a_{33} + a_{21} a_{23} a_{32} + a_{22} a_{21} a_{33}
 - a_{22} a_{23} a_{31} - a_{23} a_{32} a_{21} + a_{23} a_{31} a_{22}\ts.
$$
The bijection of the proof inserts the vertex~$2$ after the head~$1$.  For example,
the cycle decomposition \ts $\bC = (1\ts 2)(3) \in \cS_3$ \ts has weight \ts
$Q_{12}\ts Q_{21}\ts Q_{33} = a_{22}\ts a_{21}\ts a_{33}$ \ts and \ts
$\sign(\bC) = -1$, and it maps to the clow sequence \ts $(1\ts 2\ts 2)(3)$ \ts
of weight $a_{12}\cdot a_{22}\ts a_{21}\ts a_{33}$ and sign $+1$.
Carrying out all six gives \ts $R_3 = -\Mdet(Q)$, where
$$
 Q \, = \, \begin{pmatrix} a_{21} & a_{21} & a_{31} \\ a_{22} & a_{22} & a_{32} \\
 a_{23} & a_{23} & a_{33} \end{pmatrix}
$$
is obtained from~$A$ by setting the first column equal to the second.
By Lemma~\ref{t:cyc-rq}, we have \ts $\Mdet(Q) = \Cdet(Q)$ \ts in \ts
$\cB_3\ts$.  Finally, \ts $\Cdet(Q) = 0$ \ts in \ts $\cB_3\ts$, since the
first and second columns of~$Q$ coincide.
\end{ex}

\smallskip

\begin{prf}[Proof of Lemma~\ref{t:clow-rq}]
Comparing the definitions \eqref{eq:Mdet} and \eqref{eq:Vdet}, it suffices
to show that the signed sum of the weights of the clow sequences in \ts
$\cC_n^{\ts\tancirc}$ \ts vanishes in \ts $\cB_n\ts$.  Observe that every
clow sequence \ts $\bC \in \cC_n^{\ts\tancirc}$ \ts has a unique type and a
unique prefix as above.  Thus the set \ts $\cC_n^{\ts\tancirc}$ \ts is
a disjoint union, over the three types and over all possible common prefixes,
of the families of clow sequences in Lemmas~\ref{l:type1}, \ref{l:type2}
and~\ref{l:type3}.  By these lemmas, the signed sum over each family
vanishes in \ts $\cB_n\ts$.  Therefore, we have:
$$
\Vdet(A) \, - \, \Mdet(A) \ = \sum_{\bC \ts\in\ts \cC_n^{\ts\tancirc}}
 \sign(\bC) \. \weight(\bC) \, = \, 0
 \quad \text{in} \ \ \cB_n\ts,
$$
as desired.
\end{prf}

\medskip

\section{$\bq$-right-quantum matrices} \label{s:bq-RQ}

In this section we prove that \ts $\Cdet_\bq\ts$, \ts $\Mdet_\bq$ \ts and \ts
$\Vdet_\bq$ \ts are equal in \ts $\cB_n^\bq\ts$
$($Theorem~\ref{t:det-eq}$)$, the case of $\bq$-RQ
matrices, which is the most general case.  The proof combines the proofs of
the corresponding statements for the $\bq$-CF and RQ cases, given in
Sections~\ref{s:CF} and~\ref{s:RQ}, respectively.

\subsection{Cayley--Moore determinant equality} \label{ss:cyc-rqij}
We are now ready to prove the most general result of this type:

\smallskip

\begin{thm}[{\rm =Lemma~\ref{lem:C-M-det}}{}] \label{t:cyc-rqij}
Let \. $A=(a_{ij}) \in \cB_n^{\bq}$ \. be a $\bq$-RQ matrix.  Then we have:
$$
\Cdet_\bq(A) \, = \, \Mdet_\bq(A).
$$
\end{thm}

\smallskip

When all $q_{ij}=1$, we have \ts $\weight_\bq(\bC)=\weight(\bC)$, and
Theorem~\ref{t:cyc-rqij} becomes Lemma~\ref{t:cyc-rq}.  As in $\S$\ref{ss:cyc-rq}, the
summation is taken over all $p$-sequences.
\smallskip

\begin{prf}[Proof of Theorem~\ref{t:cyc-rqij}]
We repeat the bijective scheme of the proof of Lemma~\ref{t:cyc-rq}, now keeping track of
the $\bq$-weight.  Recall the map~$\Psi$ on $q$-sequences defined as follows.
Given a $q$-sequence whose longest
$p$-sequence prefix consists of~$x$ steps, with~$i$ the height at the end of step~$x$,
the map~$\Psi$ swaps the unique step starting at height~$i$ with its immediate left neighbor,
raising the rank by~$1$.

We carry out the swap together with a partner $q$-sequence, exactly as before: the partner has
the same rank and agrees with the first everywhere except that the ends of the two swapped steps
are interchanged, so if the first corresponds to a permutation~$\sigma$, the partner corresponds
to $\sigma' = \sigma\cdot\tau$ and the two carry opposite signs, $\sign\sigma' = -\sign\sigma$.
The two swapped steps are adjacent factors $a_{\ell j}\ts a_{ki}$ with $i \ne j$; say $i<j$.  By the
computation in the proof of Lemma~\ref{l:interchange}, interchanging $a_{\ell j}\ts a_{ki}$ multiplies
the $\bq$-weight~\eqref{eq:bq-weight-word} of each sequence by
$$
 c \, = \, q_{ij}\ts q_{k\ell}^{-1}\ \ (k<\ell), \qquad
 c \, = \, q_{ij}\ts q_{\ell k}\ \ (k>\ell), \qquad
 c \, = \, q_{ij}\ \ (k=\ell),
$$
and in each of these three cases relations~\eqref{eq:bq-RQ} give
$a_{\ell j}\ts a_{ki} = c\,a_{ki}\ts a_{\ell j}$ in the two summed sequences
at once.  Because the two partners carry opposite signs, the relations apply with the signs
intact, with the factor~$c$ cancelling against the $\bq$-weight exactly as in
Lemma~\ref{l:interchange}:
$$
 \mp\,\weight_\bq(\alpha) \, \pm \, \weight_\bq(\beta) \ = \
 \mp\,\weight_\bq(\Psi\alpha) \, \pm \, \weight_\bq(\Psi\beta)\ts.
$$

When a sequence is already a $p$-sequence, $\Psi$ leaves it unchanged.  Iterating the map
therefore rewrites the $\bq$-Cayley determinant in \ts $\cB_n^\bq$ \ts as the
$\bq$-Moore determinant, proving Theorem~\ref{t:cyc-rqij}.
\end{prf}

\smallskip

\subsection{Moore--Valiant determinantal equality} \label{ss:clow-rqij}

\begin{thm}[{\rm =Lemma~\ref{lem:M-V-det}}{}] \label{t:clow-rqij}
Let \. $A=(a_{ij}) \in \cB_n^{\bq}$ \. be a $\bq$-RQ matrix.  Then we have:
$$
\Mdet_\bq(A) \, = \, \Vdet_\bq(A).
$$
\end{thm}

The proof occupies the rest of this section. We carry over the double induction of $\S$\ref{ss:cancellation} verbatim, replacing
$\weight(\bC)$ by $\weight_\bq(\bC)$ throughout.  The reduction of the support to an interval
$\{1, 2, \ldots, k\}$, the Inductive Assumption~\ref{ia:types}, and the type-$1$, type-$2$,
type-$3$ grouping all apply unchanged.

\begin{lemma} \label{l:type1-qij}
For any integer $n \ge 1$, the signed $\bq$-weighted sum of the clow sequences of type~$1$
supported on a fixed multiset with $n$ elements that share a common prefix
$$
 C_1\ts C_2 \cdots (h\ts t_1\ts t_2 \cdots t_k\ts e' \cdots) \cdots
$$
vanishes in \ts $\cB_n^\bq\ts$.
\end{lemma}

\begin{prf}
We use the analogue of Inductive Assumption~\ref{ia:types} for the $\bq$-weighted sums.  Reading
the arcs $h\to t_1\to\cdots\to t_k\to e'$, every term factors at the word level as
$$
 \weight(\bC) \, = \, \Pi\, P\,\si(\bC)\,, \qquad
 \Pi=\weight(C_1)\weight(C_2)\cdots\,,\qquad
 P=a_{h t_1}\,a_{t_1 t_2}\cdots a_{t_k\ts e'}\,,
$$
with $\Pi$ (the finished clows) and $P$ (the path into~$e'$) common to all terms and $\si(\bC)$
the tail, the arc leaving~$e'$ onward.  Deleting $\Pi$ and the block $t_1\cdots t_k$ turns~$\bC$ into
a type-$3$ sequence $\ov{\bC}$ of order $m<n$, with defected clow $(h\ts e'\cdots)$ and
$\weight(\ov{\bC})=a_{h e'}\,\si(\bC)$.

By the Factorization Lemma~\ref{l:factorization}, the $\bq$-weight factors over clows and inversions between distinct clows cancel.
Sorting the surviving inversions of \ts $\weight_\bq(\bC)$ \ts by the blocks \ts $\Pi,P,\si(\bC)$ \ts gives:
\begin{equation} \label{e:type1-factor}
 \weight_\bq(\bC) \, = \, [\weight_\bq(C_1)\,\weight_\bq(C_2)\cdots]\;\cdot\;P\;\cdot\;\gamma\;\cdot\;\weight_\bq\bigl(\si(\bC)\bigr),
\end{equation}
where \ts $\weight_\bq(C_j)$ \ts and \ts $\weight_\bq(\si(\bC))$ \ts
are the $\bq$-weights of the finished clows and the tail, and
$\gamma$ is the \emph{pure scalar}
\begin{equation} \label{e:gamma-type1}
 \gamma \, = \, \prod_{x \in \set{t_1, \ldots, t_k,\, e'}} q^{-1}_{h\ts x}\;\cdot\;
 \prod_{\substack{y \in S\\ y < e'}} q^{-1}_{y\ts e'}\,,
\end{equation}
with $S$ the multiset of tail vertices.  In other words, here we remove
the support of the finished-clow vertices and one copy each of
\ts $h, t_1, \ldots, t_k, e'.$  Indeed, only the cross-block inversions with left
endpoint in~$P$ remain after Lemma~\ref{l:factorization}: the head~$h$ gives no column inversion; each
column $t_i$ cancels against its row occurrence except for $q^{-1}_{h\ts t_i}$ (the closing arc
at~$h$); and the row $e'$ of $a_{t_k\ts e'}$ leaves $q^{-1}_{h\ts e'}$ and one $q^{-1}_{y\ts e'}$ per
tail vertex $y<e'$ (whether $y$ lies in the defected clow or a later one).  No factor
of~\eqref{e:gamma-type1} depends on the order of $\si(\bC)$, so $\gamma$ is common.

By~\eqref{e:type1-factor}, with \. $\weight_\bq(C_1)\weight_\bq(C_2)\cdots$, \. $P$ \ts and \ts
$\gamma$ \ts common, we have:
$$
 \sum_{\bC} \sign(\bC)\. \weight_\bq(\bC) \, = \, [\weight_\bq(C_1)\,\weight_\bq(C_2)\cdots] \cdot P \cdot \gamma \cdot
 \sum_{\bC} \. \sign(\bC)\. \weight_\bq\bigl(\si(\bC)\bigr).
$$
The sign-preserving bijection \. $\bC\mapsto\ov{\bC}$ \. identifies the inner sum with the type-$3$ sum of
order $\ts m<n$.  As in the proof of Lemma~\ref{l:type3-qij}, we have \.
$\weight_\bq(\ov{\bC})=\gamma_0\,a_{h e'}\,\weight_\bq(\si(\bC))$, with \ts $\gamma_0,a_{h e'}$ \ts
common.  Thus, the sum above vanishes in \ts $\cB_n^\bq$ \ts by Lemma~\ref{l:type3-qij}.
Cancelling \ts $\gamma_0\,a_{h e'}\ts$, gives:
$$\sum_{\bC} \. \sign(\bC)\. \weight_\bq(\si(\bC)) \, = \, 0 \quad \ \text{in \ \, $\cB_n^\bq\..$}
$$
Multiplying on the left by \. $\weight_\bq(C_1)\weight_\bq(C_2)\cdots P\,\gamma$, we conclude that
the original sum vanishes in \ts $\cB_n^\bq\ts$, proving the lemma.
\end{prf}

\begin{ex}
\normalfont \label{ex:type1-qij}
On the support $\set{1,2,3,4,4,5,6,7,8,9}$, take the type-$1$ family whose defected clow is first
and begins $(1\ts 8\ts 2\ts 7\ts 4\,\cdots 4\cdots)$, so $h=1$, $t_1 t_2 t_3 = 8\ts 2\ts 7$, and
$e'=e=4$.  There are no finished clows, so $\Pi$ is empty and the common prefix word is
$P = a_{1\ts 8}\ts a_{8\ts 2}\ts a_{2\ts 7}\ts a_{7\ts 4}$ (the arcs $1\to 8\to 2\to 7\to 4$). Each weight word in this family begins with~$P$.
Indeed, this follows from the fact that the second~$4$
and the vertices $3,5,6,9$ can be distributed in every
admissible way.  Now, since $3$ is the only tail vertex
below~$e'=4$, formula~\eqref{e:gamma-type1} gives the common scalar
$$
 \gamma \, = \, q^{-1}_{1\ts 8}\ts q^{-1}_{1\ts 2}\ts q^{-1}_{1\ts 7}\ts q^{-1}_{1\ts 4}\;\cdot\; q^{-1}_{3\ts 4}\,.
$$
For the member
$$
 \bC\. = \. (1\ts 8\ts 2\ts 7\ts 4\ts 3\ts 5\ts 4\ts 6\ts 9)\,, \qquad
 \si(\bC) \. = \. a_{4\ts 3}\ts a_{3\ts 5}\ts a_{5\ts 4}\ts a_{4\ts 6}\ts a_{6\ts 9}\ts a_{9\ts 1}\,,
$$
one checks \. $\weight_\bq(\bC) = P\cdot\gamma\cdot \weight_\bq\bigl(\si(\bC)\bigr)$.

Hence the signed sum factors as \.
$P\.\gamma\.\cdot \. \sum_{\bC} \sign(\bC)\. \weight_\bq(\si(\bC))$, whose inner factor is the type-$3$ sum (with its prefix removed)
on $\set{1,3,4,4,5,6,9}$ and vanishes by Lemma~\ref{l:type3-qij}.
\end{ex}

\smallskip

\begin{lemma} \label{l:type2-qij}
For any integer $n \ge 1$, the signed $\bq$-weighted sum of the clow sequences of type~$2$
supported on a fixed multiset with $n$ elements that share a common prefix
$$
 C_1\ts C_2 \cdots (h\ts e' \cdots) \cdots
$$
vanishes in \ts $\cB_n^\bq\ts$.
\end{lemma}

\smallskip

The proof is completely analogous to the proof of Lemma~\ref{l:type1-qij}.

\smallskip

\begin{lemma} \label{l:type3-qij}
For any integer $n \ge 1$, the signed $\bq$-weighted sum of the clow sequences of type~$3$
supported on a fixed multiset with $n$ elements that share a common prefix
$$
 (h\ts e' \cdots) \cdots
$$
vanishes in \ts $\cB_n^\bq\ts$.
\end{lemma}

\begin{prf}
As in the proof of Lemma~\ref{l:type3}, we may assume that the support is
$\{1, 2, \ldots, k, k, \ldots, n\}$ with the single repeated pair $e' = e = k$,
and head $h = 1$.  Let
$$
 T_n \, = \, \sum_{\bC} \, \sign(\bC) \. \weight_\bq(\bC)
$$
be the signed $\bq$-weighted sum of all type-$3$ clow sequences supported on this multiset.
Every such sequence has its defected clow of the form $(h\ts e' \cdots)$, so its weight begins
with the letter $a_{h e'}$, the arc $h \to e'$.  We claim that this leading letter carries a
$\bq$-weight factor common to all terms.  Indeed, in $\weight_\bq(\bC)$ the leading position
contributes, through~\eqref{eq:bq-weight-word}, the inversions it forms with the rest of the word.  It
forms no column inversion, since its column index~$h$ is the head and hence the smallest column
present; and it forms a row inversion with each later position of smaller row, contributing
$\prod q^{-1}_{r\ts e'}$ over the rows $r < e'$ occurring after it.  As all type-$3$ sequences in
$T_n$ share the same support, this product is the same for every term.  Writing \.
$\weight_\bq(\bC) = \gamma_0\, a_{h e'}\, w_\bq(\bC)$, where \. $\gamma_0 = \prod_{r<e'} q^{-1}_{r\ts e'}$
\. is this common scalar, we factor the common left factor \. $\gamma_0\. a_{h e'}$ \. out:
$$
 T_n \, = \, \gamma_0\, a_{h e'}\, R_n\,, \qquad R_n \, = \, \sum_{\bC} \, \sign(\bC) \. w_\bq(\bC).
$$
This is the empty-block case ($k=0$) of the common scalar~$\gamma$ of
Lemma~\ref{l:type1-qij}.  Indeed, with no $t_i$ between $h$ and $e'$ and no finished clows, the prefix is
the single arc $a_{h e'}$, the head-versus-block product of~\eqref{e:gamma-type1} disappears, and
its second product \. $\prod_{y<e'} q^{-1}_{y\ts e'}$ \. becomes \ts $\gamma_0\ts.$

We claim that
$$
 R_n \, = \, -\,\Mdet_\bq(Q)\ts,
$$
where \ts $Q=(Q_{cr})$ \ts is the \ts $n\times n$ \ts matrix obtained from~$A$ by replacing
its first column with a copy of its $k$-th column:
$$
 Q_{1r} \, = \, a_{kr}\,, \qquad Q_{cr} \, = \, a_{cr} \quad \text{for all} \quad 2 \le c \le n\ts.
$$
Each entry of~$Q$ is one of the original entries~$a_{cr}$, and the columns of~$Q$ form a
sub-collection of the columns of~$A$; hence the entries of~$Q$ satisfy the same column and
crossing relations~\eqref{eq:bq-RQ}, so \ts $Q\in\cB_n^\bq$ \ts is again a $\bq$-RQ matrix.

The same sign-reversing, weight-matching bijection as in the proof of Lemma~\ref{l:type3}
proves the claim.  Indeed,
deleting the distinguished repeated~$k$ that follows the head of the defected clow turns a
type-$3$ sequence on $\{1, \ldots, k, k, \ldots, n\}$ into a cycle decomposition \ts $\bC\in\cS_n\ts$,
the split arcs $h \to k \to (\cdot)$ collapsing to a single arc whose letter belongs to the first
column of~$Q$, and conversely; the $\bq$-weights match under this bijection as in the
proof of Lemma~\ref{t:clow-cf-qij}.

By Theorem~\ref{t:cyc-rqij} applied to~$Q$, we have \ts
$\Mdet_\bq(Q) = \Cdet_\bq(Q)$ \ts in \ts $\cB_n^\bq\ts$.
The first and the $k$-th columns of~$Q$ are equal, both being the $k$-th column of~$A$, and the
$\bq$-Cayley determinant of a $\bq$-RQ matrix with two equal columns vanishes
in \ts $\cB_n^\bq\ts$ by Lemma~\ref{lem:equal-columns}.  We conclude:
$$
 R_n \, = \, -\,\Mdet_\bq(Q) \, = \, -\,\Cdet_\bq(Q) \, = \, 0
 \quad \text{in} \ \ \cB_n^\bq\..
$$
Multiplying by the common left factor \. $\gamma_0\, a_{h e'}$,
and noting that scalar~$\gamma_0$ is nonzero, we conclude that
\ts $T_n = \gamma_0\, a_{h e'}\, R_n$ \ts vanishes in \ts $\cB_n^\bq\ts$.
\end{prf}

Together, Lemmas~\ref{l:type1-qij}, \ref{l:type2-qij} and~\ref{l:type3-qij} show that the signed
$\bq$-weighted sum of all clow sequences in \ts $\cC_n^{\ts\tancirc}$ \ts vanishes in \ts
$\cB_n^\bq\ts$, which proves Theorem~\ref{t:clow-rqij}.
\smallskip

{\small
\begin{rem}
\normalfont
When all $q_{ij}\gets 1$, the whole section specializes to $\S$\ref{ss:clow-rq}. Indeed,
the relations \eqref{eq:bq-RQ} become \eqref{eq:rq}, the $\bq$-weight $\weight_\bq$ reduces to
$\weight$, and Theorems~\ref{t:cyc-rqij} and~\ref{t:clow-rqij} become Lemmas~\ref{t:cyc-rq}
and~\ref{t:clow-rq}.  Setting all $q_{ij}\gets q$ recovers the $q$-right-quantum case,
but since the proof do not significantly simplify we opted for the general case.
\end{rem}
}

\medskip


\section{Final remarks and open problems} \label{s:finrem}

\subsection{}\label{ss:finrem-Pfaff}
For a \emph{skew-symmetric matrix} \ts $A=-A^T$, computing the Pfaffian plays
the same role as the determinant of the usual matrix.  Numerically,
this is an equivalent problem since \ts $\det(A)=\pf(A)^2$.  However, computing
the Pfaffian as a polynomial of matrix entries is different from computing
the permanent.  A variation on the Mahajan--Vinay algorithm was given in \cite{MSV04},
see also \cite{Rote01,Urb10}.  It would be interesting to see if this algorithm
can be extended to compute the Pfaffian of right-quantum matrices.

\subsection{}\label{ss:finrem-CH}
In~\cite{MV99}, Mahajan and Vinay wrote that their algorithm was partly inspired
by combinatorial proofs of the \emph{Cayley--Hamilton theorem} (CHT) by Rutherford,
Straubing and Zeilberger \cite{Rut64,Str83,Zei85}.  Most recently,
Ikenmeyer \cite{Ike25} constructed a new ABP for the
determinant.  He suggested that his approach is even closer to Straubing's and
Zeilberger's combinatorial arguments.  It would be interesting to see if there is a proof
of our Main Theorem~\ref{t:main} in the style of Ikenmeyer's construction,
as it would potentially avoid using our involved combinatorial arguments.

In fact, several (possibly, equivalent) versions of a $\bq$-CHT are already present
or implicit in the literature.  Indeed, both the RQ and $q$-RQ versions of the CHT
are given in \cite{CF08,CFR09,CFRS14}.  For $\bq$-RQ matrices, one can take
a fully noncommutative version of the CHT by Gelfand and Retakh for quasideterminants
\cite{GR92,GGRW05}, translated into the language of $\bq$-Cayley determinants
following \cite{GR91,ER99,KP07}.
Alternatively, one can use the ``$\bq=1$'' principle by Konvalinka--Pak \cite[Lemma~12.3]{KP07},
generalizing the ``$q=1$'' principle by Foata--Han \cite{FH08}, to derive the $\bq$-CHT
directly from the RQ version.

\subsection{}\label{ss:finrem-imm}
For an irreducible character \ts $\chi^\la: S_n \to \zz$, where \ts $\la$ \ts
is a partition of~$n$, an \defn{immanant} \ts of the matrix \ts $A=(a_{ij})$ \ts
is defined as follows:
$$
\Imm_\la(A) \, := \, \sum_{\si \in S_n} \. \chi^\la(\si) \, a_{1 \si_1} \. \cdots \. a_{n \si_n}\..
$$
We refer to \cite{GJ92} for an accessible introduction to immanants and
their applications to algebraic combinatorics.  Note that the immanants
generalize both the permanent and the determinant, for the trivial and
sign characters, respectively.  The algorithmic and complexity properties
of immanants have also been extensively studied, see \cite[$\S$7]{Bur00-book}
for an overview, and \cite{Cur21,MR26} for some notable recent results.

Given the negative results in the papers cited above, it would be
interesting to see if one could extend our approach to obtain positive
results for computing quantum immanants on quantum matrices.  As a starting
point, we refer to Okounkov \cite{Oko96} and Konvalinka \cite{Kon12}, for
two different definitions of quantum immanants.

\subsection{}\label{ss:finrem-R}
Arguably, matrix inversion is the real reason the noncommutative determinant
can be computed efficiently.  This is the philosophy underlying the theory
of quasideterminants, see \cite{GGRW05}.  Matrix inversion exists, in particular,  
when there is an $R$-matrix which satisfied the quantum Yang--Baxter equation, 
see \cite{ESS00,ER99}.  We speculate that the determinantal 
identities in Theorem~\ref{t:det-eq} are in fact consequences of algebraic 
properties guaranteed by the $R$-matrices.  More concretely, we conjecture that
Theorem~\ref{t:det-eq} can be extended to all generalized Belavin--Drinfeld
triples \cite{ESS00}, see also \cite{Hod15} for the introduction.  

\vskip.7cm

{\small
\subsection*{Acknowledgements}
We are grateful to Alexander Barvinok, Pavel Etingof, 
Christian Ikenmeyer, Pasha Pylyavskyy, Vladimir Retakh and
Avi Wigderson, for interesting discussions and helpful remarks.
The first author (IP) was partially supported by the NSF grant CCF-2302173.

\vskip.5cm

\subsection*{Use of AI}
During the preparation of this manuscript, the authors used Claude
(Anthropic, Claude Opus~4.8) to make numerous experiments in the
combinatorics of words, in our search for the desired sign-reversing
involution.  We also used it to make figures and proof check the examples.
Additionally, we used ChatGPT (OpenAI, GPT-5.5) for historical
background and references.  All AI-generated suggestions were
reviewed and, where appropriate, modified before inclusion in
the manuscript. More often than not, they proved false or unhelpful,
and were simply ignored.
}
\medskip

\newpage

\vskip.5cm

\end{document}